\documentclass[graybox]{SNmult}

\usepackage{type1cm}        % activate if the above 3 fonts are
\usepackage{makeidx}         % allows index generation
\usepackage{graphicx}        % standard LaTeX graphics tool
\usepackage{multicol}        % used for the two-column index
\usepackage[bottom]{footmisc}% places footnotes at page bottom

\usepackage{newtxtext}       % 
\usepackage[varvw]{newtxmath}       % selects Times Roman as basic font

\makeindex             % used for the subject index
\usepackage[normalem]{ulem}

\usepackage{hyperref}
\usepackage{dsfont, cancel}
\usepackage{enumitem}

\usepackage{booktabs} % For nice-looking tables

\usepackage[
backend=biber,
style=numeric,
giveninits=true,
sorting=nyt,
maxnames=4,
abbreviate=true
]{biblatex}
\AtEveryBibitem{
    \clearfield{urldate}
    \clearfield{keywords}
    \clearfield{urlyear}
    \clearfield{urlmonth}
    \clearfield{url}
    \clearfield{issn}
    \clearfield{isbn}
    \clearfield{pages}
    \clearfield{doi}
}

    \newtheorem{assump}[theorem]{Assumption}

\let\pa\partial

\newcommand{\wt}{\widetilde}

\renewcommand{\b}{\mathrm{b}}
\renewcommand{\d}{\mathrm{d}}
\renewcommand{\D}{\mathrm{D}}

\newcommand{\dt}{\,\mathrm{d}t}
\newcommand{\ds}{\,\mathrm{d}s}
\newcommand{\dx}{\,\mathrm{d}\mathbf{x}}

\newcommand{\ba}{\mathbf{a}}
\newcommand{\bb}{\mathbf{b}}
\newcommand{\bc}{\mathbf{c}}
\newcommand{\bd}{\mathbf{d}}
\newcommand{\be}{\mathbf{e}}

\newcommand{\bh}{\mathbf{h}}
\newcommand{\bm}{\mathbf{m}}

\newcommand{\bu}{\mathbf{u}}
\newcommand{\bv}{\mathbf{v}}
\newcommand{\bw}{\mathbf{w}}
\newcommand{\bx}{\mathbf{x}}

\newcommand{\bzero}{\mathbf{0}}

\newcommand{\bB}{\mathbf{B}}
\newcommand{\bD}{\mathbf{D}}
\newcommand{\bE}{\mathbf{E}}
\newcommand{\bF}{\mathbf{F}}
\newcommand{\bH}{\mathbf{H}}
\newcommand{\bJ}{\mathbf{J}}
\newcommand{\bL}{\mathbf{L}}
\newcommand{\bM}{\mathbf{M}}
\newcommand{\bN}{\mathbf{N}}

\newcommand{\bR}{\mathbf{R}}
\newcommand{\bU}{\mathbf{U}}
\newcommand{\bV}{\mathbf{V}}
\newcommand{\bW}{\mathbf{W}}

\newcommand{\obU}{\overline{\mathbf{U}}}

\newcommand{\obM}{\overline{\mathbf{M}}}
\newcommand{\obH}{\overline{\mathbf{H}}}
\newcommand{\oP}{\overline{P}}
\newcommand{\oPhi}{\overline{\Phi}}
\newcommand{\oPsi}{\overline{\Psi}}

\newcommand{\cB}{\mathcal{B}}
\newcommand{\cC}{\mathcal{C}}

\newcommand{\sD}{\mathscr{D}}
\newcommand{\sE}{\mathscr{E}}
\newcommand{\sG}{\mathscr{G}}
\newcommand{\sM}{\mathscr{M}}
\newcommand{\sR}{\mathscr{R}}

\newcommand{\bbI}{\mathbb{I}}
\newcommand{\bbP}{\mathbb{P}}
\newcommand{\bbR}{\mathbb{R}}
\newcommand{\bbU}{\mathbb{U}}
\newcommand{\bbX}{\mathbb{X}}

\newcommand{\dL}{\mathds{L}}

\newcommand{\cT}{\mathcal{T}}

\newcommand{\cS}{\mathcal{S}}
\newcommand{\cR}{\mathcal{R}}

\newcommand{\sfe}{\mathsf{e}}
\newcommand{\sfE}{\mathsf{E}}
\newcommand{\sfU}{\mathsf{U}}
\newcommand{\osfU}{\overline{\mathsf{U}}}
\newcommand{\sfu}{\mathsf{u}}

\renewcommand{\div}{\operatorname{div}}
\newcommand{\curl}{\operatorname{curl}}
\newcommand{\bcurl}{\operatorname{\mathbf{curl}}}

\newcommand{\conv}[1]{(#1 \cdot \nabla) #1}
\newcommand{\Conv}[2]{(#1 \cdot \nabla) #2}

\newcommand{\bn}{\mathbf{n}}

\begin{document}
% \motto{On the occasion of G. Savar\'e's 60th birthday.}
\title*{Ferrofluids: Modeling and Approximation}
% Use \titlerunning{Short Title} for an abbreviated version of
% your contribution title if the original one is too long
\author{Gonzalo A. Benavides, Ricardo H. Nochetto, and Mansur Shakipov}
% Use \authorrunning{Short Title} for an abbreviated version of
% your contribution title if the original one is too long
\institute{%Name of First Author \at Name, Address of Institute, \email{name@email.address}
Gonzalo A. Benavides\orcidID{0000-0002-6072-0427}, Mansur Shakipov\orcidID{0000-0003-4837-6747} \at Department of Mathematics, University of Maryland, College Park, MD 20742, USA, \email{gonzalob@umd.edu}, \email{shakipov@umd.edu}
\and Ricardo H. Nochetto\orcidID{0009-0002-8895-1144} \at Department of Mathematics and Institute for Physical Science and Technology, University of Maryland, College Park, MD 20742, USA, \email{rhn@umd.edu}
% \and Mansur Shakipov \at Department of Mathematics, University of Maryland, College Park, MD 20742, USA, \email{shakipov@umd.edu}
}
%
% Use the package "url.sty" to avoid
% problems with special characters
% used in your e-mail or web address
%
\maketitle
\abstract{Starting from Maxwell's and linear momentum balance equations, we derive a ferrofluid model using the generalized Onsager's principle. 
Guided by a discrete perturbation estimate, we design and analyze families of Galerkin schemes that converge to sufficiently regular solutions and derive error estimates.  Finally, we numerically explore the model with our proposed method.
}

\section{Introduction}\label{sec:intro}
    Ferrofluid is a colloidal liquid wherein ferromagnetic particles are suspended in a carrier liquid such as water \cite{Rosensweig2013}.
    By design, it behaves as a liquid and responds to an applied magnetic field.
    This material has broad applications in the design of loudspeakers, cell separation, and audio visualization, among other areas \cite{ferrofluids-review, SocoliucEtAl2022}.
    In this work, we consider the following reduced ferrofluid system:
    \begin{equation} \label{eq:ff}
        \begin{cases}
            \D_t \bu - \nu \Delta \bu + \nabla p = \mu_0 (\bm \cdot \nabla) \bh, \qquad \D_t \bm = - \frac{1}{\tau} (\bm - \chi \bh), \\
            \div \bu = 0, \qquad \bcurl \bh = \bzero, \qquad \bb = \mu_0(\bm + \bh), \qquad  \div \bb = 0.
        \end{cases}
    \end{equation}
    Here, $\D_t := \pa_t + (\bu \cdot \nabla)$ is the material derivative, $(\bu, p)$ is the velocity-pressure pair of a fluid flow, $(\bm, \bh, \bb)$ are, respectively, the magnetization, the magnetic H-field, and the induction B-field. The constants $(\nu, \mu_0, \tau, \chi)$ are, respectively, the kinematic fluid viscosity, vacuum permeability, relaxation time, and magnetic susceptibility. System \eqref{eq:ff} is a simplification of the full Rosensweig system \cite{Rosensweig2013} in that it ignores the effect of the angular momentum \cite[\S1.2]{NochettoTrivisaWeber2019}, whilst retaining the key terms like the Kelvin force $\mu_0(\bm \cdot \nabla) \bh$ and the convective term $(\bu \cdot \nabla)\bm$ in the magnetization equation. We note that the full Rosensweig system admits suitable global weak solutions \cite[\S3]{NochettoTrivisaWeber2019}.
    
    System \eqref{eq:ff} is supplemented with initial conditions $\bu(0) = \bu_0$, $\bm(0) = \bm_0$, $\bh(0) = \bh_0$, boundary conditions $\bu = \bzero$, $(\bm + \bh) \cdot \bn_{\pa \Omega} = \bh_a \cdot \bn_{\pa \Omega}$, and the uniqueness condition $\int_\Omega p \dx = 0$.
    A solution $(\bu, p, \bm, \bh)$ of \eqref{eq:ff} formally obeys the following energy law:
    \begin{equation} \label{eq:ff-energy}
        \frac{\d}{\dt} \left[\frac12 \|\bu\|^2 + \frac{\mu_0}{2 \chi} \|\bm\|^2 + \frac{\mu_0}{2} \|\bh\|^2 \right] + \nu \|\nabla \bu\|^2 + \frac{\mu_0}{\tau \chi} \|\bm - \chi \bh\|^2 = \mu_0 \int_\Omega \pa_t \bh_a \cdot \bh \dx,
    \end{equation}
    where $\|\cdot\|$ denotes the $L^2(\Omega)$-norm.
    Equality \eqref{eq:ff-energy} is obtained by testing the momentum equation with $\bu$ and the magnetization equation with $\frac{\mu_0}{\chi}(\bm - \chi \bh)$ \cite[\S3.1]{NochettoTrivisaWeber2019}.
    
    The analysis of the Rosensweig model \eqref{eq:ff} is notoriously difficult, even in the reduced case. The primary reason for that is the fact that the energy law \eqref{eq:ff-energy} only ensures control of $\bm$ and $\bh$ in $L^\infty(0, T; \bL^2(\Omega))$, which is insufficient to make sense of or control the Kelvin force $\mu_0 (\bm \cdot \nabla) \bh$ as is. This is the reason why, in \cite{NochettoTrivisaWeber2019}, to establish the existence of weak solutions, the authors had to rewrite the Kelvin force term and resort to the notion of renormalized solutions. This is also reflected in the numerical analysis literature---convergence results are only available either for regularized solutions (where $- \sigma \Delta \bm$ is added to the magnetization equation for small $\sigma > 0$) \cite{WuXie2023,WuXie2025,DongHuangHuang2025,DongHuangHuangTang2025,KeramHuangHe2026}, or assuming that $\bh$ is given (i.e.~it is not an unknown) \cite{NochettoSalgadoTomas2016, NochettoSalgadoTomas2016-2, MaoSun2025}. 
    
    We stress that circumventing $\sigma \Delta \bm$ is not merely a matter of sharper analysis. This singularly perturbed term is added to control the nonlinearities involving derivatives of $\bm$ that is otherwise only bounded in $L^\infty(0,T; \bL^2(\Omega))$. This allows using the usual $H^1$-conforming elements at the expense of a worsened stability constant appearing in the error estimate: it behaves like $\exp (C(1+T)/\sigma^k)$ for some $k > 0$, which comes from Gr\"onwall's inequalities. As a result, no polynomial scaling between $h$ and $\sigma$ can yield convergence to solutions of \eqref{eq:ff} with $\sigma = 0$, rendering algorithms based on this approach that have theoretical guarantees unfeasibly expensive.

    With $\sigma = 0$, the energy law \eqref{eq:ff-energy} indicates that deriving an error estimate for \eqref{eq:ff} would amount to moving all the derivatives away from $\bm$ and $\bh$ using integration by parts. Given the complicated structure of the system, to carry out this goal, we had to restrict ourselves to $d=2$ and use a very special numerical method that involves $H^2$-conforming elements and exactly divergence-free velocities.
    Our approach extends to $d=3$ albeit the modification is by no means obvious \cite{BenavidesNochetto2026-ferrofluids}.
    This is the price of being faithful to the physically realistic model with $\sigma = 0$, and also what delivers, to the best of our knowledge, the first error estimates in this setting.

    The goal of this work is twofold:
    \begin{enumerate}
        \item We derive the ferrofluid system \eqref{eq:ff} using the generalized Onsager's principle; this is the content of section \ref{sec:derivation}. This gives rise to the energy law \eqref{eq:ff-energy}.
        \item We introduce two potentials to represent the $\bcurl$-free $\bh$ and the $\div$-free $\bb$ fields to design and analyze certain families of Galerkin schemes that converge to sufficiently regular solutions of \eqref{eq:ff} in section \ref{sec:discretization}. This yields error estimates.
    \end{enumerate}
    We conclude this work with numerical experiments in section \ref{sec:experiment}.

\section{Derivation of the ferrofluid system} \label{sec:derivation}
    Let $\Omega \subseteq \bbR^3$ be an open, bounded, and simply-connected set. The goal of this section is to derive the ferrofluid equations \eqref{eq:ff} using Onsager's principle \cite{Onsager1931} starting from Maxwell's and the balance of linear momentum equations. 

    \subsection{Onsager's principle, energy law, and constraints} \label{sec:onsager}
        Onsager's principle requires postulating a general form of evolution PDEs, and defining a priori an \textit{energy} $\sE$ and a \textit{dissipation functional} $\sD$ \cite{Doi2011, Doi2015}.
        To illustrate this idea, for state variables $\bv_j$ of the underlying thermodynamic system, we define the energy
        \[
        \sE = \sum_j \alpha_j \|\bv_j\|^2, \qquad \alpha_j > 0.
        \]
        We postulate the following evolution PDEs:
        \begin{equation} \label{eq:ev-eq}
            \pa_t \bv_j = \bF_j^d + \bF_j^r + \bF_j^{\mathrm{ext}}.
        \end{equation}
        where $\bF_j^d$ are called the \textit{dissipative} quantities, $\bF^r_j$ are called the \textit{reversible} quantities, and $\bF_j^{\mathrm{ext}}$ account for environmental effects. The former must contribute to the rate of change of the energy via the dissipation functional $\sD$ given, e.g., as
        \[
        \sD = \sum_j \beta_j \|\bF_j^d\|^2, \qquad \beta_j > 0.
        \]
        The reversible quantities $\bF_j^r$, on the other hand, must not contribute to the rate of change of the energy and, in principle, may admit multiple different expressions, so a judicious choice needs to be made.
        
        Onsager's principle allows deriving explicit expressions for $\bF_j^d$ and $\bF_j^r$ that yield the following energy law: $\frac{\d}{\dt} \sE + 2 \sD = 2 \sum_j \alpha_j \big(\bF_j^{\mathrm{ext}}, \bv_j \big)$. To see this, let us first compute $\frac{\d}{\dt} \sE$ using \eqref{eq:ev-eq}:
        \[
        \frac{\d}{\dt} \sE = \sum_j 2 \alpha_j \big( \pa_t \bv_j, \bv_j \big) = \sum_j 2 \alpha_j \big(\bF_j^d + \bF_j^r + \bF_j^{\mathrm{ext}}, \bv_j \big).
        \]
        Here, $\big(\cdot, \cdot \big)$ denote the $L^2(\Omega)$-inner product. Onsager's principle dictates that $\bF_j^d$ are stationary points of the \textit{Rayleghian} $\sR := \frac{\d}{\dt} \sE + \sD$, namely
        \[
        \delta_{\bF_j^d} \sR = 0 \implies 2 \alpha_j \bv_j + 2 \beta_j \bF_j^d = 0,
        \]
        whence
        \[
        \bF_j^d = - \frac{\alpha_j}{\beta_j} \bv_j, \quad \text{and} \quad \sD = \sum_j \frac{\alpha_j^2}{\beta_j} \|\bv_j\|^2.
        \]
        On the other hand, the reversible terms must vanish so as not to affect the rate of change $\frac{\d}{\dt} \sE$ of energy:
        \[
        \sum_{j \in J} \alpha_j \big( \bF_j^r, \bv_j \big) = 0.
        \]
        These two relations yield the claimed energy law:
        \begin{equation} \label{eq:abst-en-law}
            \frac{\d}{\dt} \sE = \sum_j 2 \alpha_j \big(\bF_j^d, \bv_j \big) + 2 \alpha_j \big(\bF_j^{\mathrm{ext}}, \bv_j \big) = - 2 \sD + 2 \sum_j \alpha_j \big(\bF_j^{\mathrm{ext}}, \bv_j \big).
        \end{equation}
        Identity \eqref{eq:abst-en-law} corresponds to an energy law for an open system, which sees external data. If the system is closed instead, i.e., $\bF_j^{\mathrm{ext}} = 0$, then we obtain the more familiar relation $\frac{\d}{\dt} \sE = - 2 \sD$.
        
        Recall that the ferrofluid system \eqref{eq:ff} includes three constraints: $\div \bu = 0$, $\bcurl \bh = \bzero$, and $\div \bb = 0$.
        In order to incorporate them into Onsager's formalism, we need to artificially introduce evolution equations for the constraints' Lagrange multipliers. This is because each state variable $\bu$, $p$, $\bm$, $\bh$ and $\bb$, must have its own evolution equation \cite[Section 3]{Doi2011}. For instance, we shall replace $\div \bu = 0$ in \eqref{eq:ff} with $\varepsilon \D_t p + \div \bu = 0$ for $\varepsilon > 0$, and then formally take $\varepsilon \downarrow 0$ once the system is derived.
        For the remaining constraints, instead of introducing artificial transience, we shall consider Maxwell's equations and resort to an asymptotic approximation.
    
    \subsection{Linear momentum balance and Maxwell's equations}
        We postulate the momentum balance and relaxed continuity equations read as follows:
        \begin{equation*} \label{eq:ns-art-compr}
            \D_t \bu = \bF_\bu^r + \div \bF_\bu^d, \qquad \varepsilon \D_t p = \div \bF_p^r.
        \end{equation*}
        We do not include dissipative quantities in the relaxed continuity equation because the latter should not contribute to the energy.
        The macroscopic version of Maxwell's equations \cite{Monk2003} reads:
        \begin{equation*} \label{eq:maxwell-gen}
            \pa_t \bd = \bJ_f + \bcurl \bh, \qquad \div \bd = \rho_f, \qquad \pa_t \bb = - \varepsilon_0^{-1} \bcurl \bd, \qquad \div \bb = 0,
        \end{equation*}
        where $\bd$ is the electric displacement field, $\rho_f$ is the free electric charge density, $\bJ_f$ is the free current density, and $\varepsilon_0 \approx 8.854 \times 10^{-12} \frac{\mathrm{farad}}{\mathrm{meter}}$  is the vacuum permittivity. For this work, we shall consider no electric charge and no free current, thereby leading to
        \begin{equation*} \label{eq:maxwell}
            \pa_t \bd = \bcurl \bh, \qquad \div \bd = 0, \qquad \pa_t \bb = - \varepsilon_0^{-1} \bcurl \bd, \qquad \div \bb = 0,
        \end{equation*}
        To incorporate the constraints $\div \bb = 0$ and $\div \bd = 0$, we introduce two Lagrange multipliers $\lambda_\bb, \lambda_\bd$, and one artificial evolution equation for each of them, thereby yielding
        \begin{equation*}
        \begin{cases}
            \pa_t \bd = \bcurl \bh - \nabla \lambda_\bd, \qquad \varepsilon \pa_t \lambda_\bd = - \div \bd, \\
            \pa_t \bb = - \varepsilon_0^{-1} \bcurl \bd + \nabla \lambda_\bb, \qquad \varepsilon \pa_t \lambda_\bb = - \div \bb.
        \end{cases}
        \end{equation*}
        Lastly, we assume that magnetization is transported by the flow and has no reversible quantities, hence
        \begin{equation*}
            \D_t \bm = \bF_\bm^d.
        \end{equation*}
        To close this system, we add the constitutive equation
        \begin{equation*}
            \bb = \mu_0(\bm + \bh).
        \end{equation*}
        The full postulated ferrofluid system then reads:
        \begin{equation} \label{eq:ff-gen-1}
            \begin{cases}
                \D_t \bu = \bF_\bu^r + \div \bF_\bu^d, \qquad \varepsilon \D_t p = \div \bF_p^r, \qquad \D_t \bm = \bF_\bm^d \\
                \pa_t \bd = \bcurl \bh - \nabla \lambda_\bd, \qquad \varepsilon \pa_t \lambda_\bd = - \div \bd, \qquad \bb = \mu_0(\bm + \bh) \\
                \pa_t \bb = - \varepsilon_0^{-1} \bcurl \bd + \nabla \lambda_\bb, \qquad \varepsilon \pa_t \lambda_\bb = - \div \bb.
            \end{cases}
        \end{equation}
        We assume the following boundary conditions: $\bu = \bzero$, and $(\bm + \bh) \cdot \bn_{\pa \Omega} = \bh_a \cdot \bn_{\pa \Omega}$, where $\bh_a$ is an applied magnetic field.
        It is useful to define a harmonic extension of $\bh_a \cdot \bn_{\pa \Omega}$ into the whole domain $\Omega$ as follows: $\bh_a := \nabla \phi_a$, where $\phi_a$ satisfies
        \begin{equation} \label{eq:harm-ext}
            -\Delta \phi_a = 0 \text{ in } \Omega, \qquad \nabla \phi_a \cdot \bn_{\pa \Omega} = \bh_a \cdot \bn_{\pa \Omega} \text{ on } \pa \Omega, \qquad 
            % \textstyle\int_\Omega \phi_a \dx
            (\phi_a,1)= 0.
        \end{equation}
        One easily verifies that $\div \bh_a = \div \nabla \phi_a = 0$ and $\bcurl \bh_a = \bcurl \nabla \phi_a = \bzero$.
        
        We need to determine the dissipative quantities $\bF_\bu^d, \bF_\bm^d$, as well as the reversible quantities $\bF_\bu^r, \bF_p^r$. However, before we proceed with that calculation, we shall first simplify \eqref{eq:ff-gen-1}.
        
    \subsection{Eliminating Lagrange multipliers and displacement field}
        Assume that the initial condition for $\bd$ is divergence-free.
        Moreover, under the premise that electric effects are negligible compared to magnetic effects, we set $\bd \cdot \bn_{\pa\Omega} = 0$ on $\partial\Omega$.
        We want to show that $\lambda_\bd = 0$. Let us focus on the two equations:
        \begin{equation}\label{eq:lambdad-two-eqs}
            \pa_t \bd = \bcurl \bh - \nabla \lambda_\bd, \qquad \varepsilon \pa_t \lambda_\bd = - \div \bd.
        \end{equation}
        Taking the divergence of the first equation above and using $\div \bd = - \varepsilon \pa_t \lambda_\bd$, we obtain the wave equation for $\lambda_\bd$:
        \[
        \varepsilon \pa_{tt} \lambda_\bd = \Delta \lambda_\bd.
        \]
        Since $\div\bd(0)=0$, the Lagrange multiplier $\lambda_\bd$ should satisfy $\lambda_\bd(0) = 0$ and $\pa_t \lambda_\bd(0) = \div \bd(0) = 0$.
        Next, since the boundary flux of the first equation in \eqref{eq:lambdad-two-eqs} should be unaffected by $\lambda_\bd$, 
        we assume that $\bn_{\pa \Omega} \cdot \nabla \lambda_\bd = 0$ on $\pa \Omega$.
        The uniqueness of the wave equation allows us to deduce that $\lambda_\bd \equiv 0$; this implies $\div \bd \equiv 0$.
        Proceeding similarly, we obtain $\lambda_\bb \equiv 0$ as well.
        
        Next, we shall use an asymptotic approximation for $\bd$.
        Using $\varepsilon_0 \pa_t \bb \approx \bzero$ because $\varepsilon_0 \approx 8.854 \times 10^{-12}$, we arrive at the simplification $\bcurl \bd = \bzero$, from which we further infer that $\bd$ is harmonic. Using that $\Omega$ is simply connected and $\bd \cdot \bn_{\pa \Omega} = 0$ on $\pa \Omega$, we deduce that $\bd \equiv \bzero$. From \eqref{eq:ff-gen-1}, we therefore arrive at the simplified system:
        \begin{equation} \label{eq:ff-gen}
            \begin{cases}
                \D_t \bu = \bF_\bu^r + \div \bF_\bu^d, \quad \varepsilon \D_t p = \div \bF_p^r, \quad \D_t \bm = \bF_\bm^d, \\
                \bb = \mu_0(\bm + \bh), \quad \curl \bh = \bzero, \quad
                \div \bb = 0.
            \end{cases}
        \end{equation}

    \subsection{Energy, dissipation functional and Rayleighian}
        We consider the following time-dependent energy $\sE_\varepsilon$, dissipation functional $\sD$, and Rayleighian $\sR_\varepsilon$ functionals:
        \begin{equation} \label{eq:energies}
            \begin{split}
                \sE_\varepsilon & := \frac{1}{2} \|\bu\|^2 + \frac{\varepsilon}{2} \|p\|^2 + \frac{\mu_0}{2 \chi} \|\bm\|^2 + \frac{\mu_0}{2} \|\bh\|^2, \\
                \sD & := \frac{1}{2 \nu} \|\bF^d_\bu\|^2 + \frac{\tau \mu_0}{2 \chi} \|\bF^d_\bm\|^2, \quad \sR_\varepsilon := \frac{\d}{\dt} \sE_\varepsilon + \sD.
            \end{split}
        \end{equation}
        The rate of change of $\sE_\varepsilon$ reads:
        \begin{equation} \label{eq:pat-sE-eps}
            \begin{split}
                \frac{\d}{\dt} \sE_\varepsilon & = \big(\pa_t \bu, \bu \big) + \varepsilon \big(\pa_t p, p \big) + \frac{\mu_0}{\chi} \big(\pa_t \bm, \bm \big) + \mu_0 \big(\pa_t \bh, \bh \big).
            \end{split}
        \end{equation}
        Note that we have explicit expressions for $\pa_t \bu, \pa_t p, \pa_t \bm$ in \eqref{eq:ff-gen}, while the evolution equation for $\bh$ is implicit: using $\bb = \mu_0(\bm + \bh)$, we deduce that
        \begin{equation} \label{eq:pat-bh}
            \pa_t \bh = \mu_0^{-1} \pa_t \bb - \pa_t \bm = \mu_0^{-1} \pa_t \bb - \bF_\bm^d + (\bu \cdot \nabla) \bm.
        \end{equation}
        It turns out that the only data term $\bF^{\mathrm{ext}}$ in our system is hidden in $\mu_0^{-1} \pa_t \bb$ in \eqref{eq:pat-bh}; compare with \eqref{eq:ev-eq}. Indeed, upon testing \eqref{eq:pat-bh} with $\nabla \wt \phi$ for a scalar potential $\wt \phi$, we find
        \begin{equation} \label{eq:pat-bb}
            \begin{split}
                \mu_0^{-1} \big(\pa_t \bb, \nabla \wt \phi \big) & = \big(\pa_t (\bm + \bh), \nabla \wt \phi \big) = \big(\pa_t \bh_a, \nabla \wt \phi \big),
            \end{split}
        \end{equation}
        where we added and subtracted $\big(\pa_t \bh_a, \nabla \wt \phi \big)$, and used $\div (\bm + \bh - \bh_a) = 0$ and $(\bm + \bh - \bh_a) \cdot \bn_{\pa_\Omega} = 0$ on $\pa \Omega$.
        Since $\bcurl \bh = \bzero$ and $\Omega$ is simply connected, we infer that $\bh = \nabla \phi$ for some scalar-valued potential $\phi$, whence
        \begin{equation} \label{eq:pat-bh-2}
            \mu_0 \big(\pa_t \bh, \bh \big) = \big(\pa_t \bh_a, \bh \big) + \big(-\bF_\bm^d + (\bu \cdot \nabla) \bm, \bh \big).
        \end{equation}
        Therefore, using \eqref{eq:ff-gen} for $\pa_t \bu, \pa_t p, \pa_t \bm$ and \eqref{eq:pat-bh-2} for $\pa_t \bh$, we rewrite \eqref{eq:pat-sE-eps} as
        \begin{equation} \label{eq:dt-cE-eps}
            \begin{split}
                \frac{\d}{\dt} \sE_\varepsilon & = \big(\bF_\bu^r - \conv \bu + \div \bF_\bu^d, \bu \big) + \big(\div \bF_p^r - \varepsilon \bu \cdot \nabla p, p \big) \\
                & + \mu_0 \chi^{-1} \big(\bF_\bm^d - \Conv{\bu}{\bm}, \bm \big) + \mu_0 \big(- \bF_\bm^d + \Conv{\bu}{\bm}, \bh \big) + \mu_0  \big(\pa_t \bh_a, \bh \big),
            \end{split}
        \end{equation}
        where the last data term is neither dissipative nor reversible and accounts for the system not being closed, as observed in \eqref{eq:abst-en-law}.
        As per \eqref{eq:energies}, the expression for the Rayleighian $\sR = \frac{\d}{\dt} \sE_\varepsilon + \sD$ thus reads
        \begin{multline} \label{eq:cR-eps}
            \sR = \big(\bF_\bu^r - \conv \bu + \div \bF_\bu^d, \bu \big) + \big(\div \bF_p^r - \varepsilon \bu \cdot \nabla p, p \big) + \frac{\mu_0}{\chi} \big(\bF_\bm^d - \Conv{\bu}{\bm}, \bm \big) \\
            \quad + \mu_0 \big(- \bF_\bm^d + \Conv{\bu}{\bm}, \bh \big) + \mu_0 \big(\pa_t \bh_a, \bh \big) + \frac{1}{2 \nu} \|\bF^d_\bu\|^2 + \frac{\tau \mu_0}{2 \chi} \|\bF^d_\bm\|^2.
        \end{multline}
        We are now in a position to determine dissipative quantities $\bF_\bu^d$ and $\bF_\bm^d$, and reversible quantities $\bF_\bu^r$ and $\bF_p^r$.
        
    \subsection{Dissipative and reversible quantities}
        As explained in section \ref{sec:onsager}, the dissipative quantities are the stationary points of the Rayleighian $\sR_\varepsilon$ given in \eqref{eq:cR-eps}.
        If $\bB : \Omega \to \bbR^{3 \times 3}$ is an arbitrary symmetric tensor-valued field, then $\bF_\bu^d$ is determined by
        \begin{equation*}
            0 = \langle \delta_{\bF^d_\bu} \sR_\varepsilon, \bB \rangle = \big(-\nabla \bu + \nu^{-1} \bF^d_\bu, \bB \big),
        \end{equation*}
        whence $\bF^d_\bu = \nu D \bu$, where $D := \frac{1}{2} (\nabla + \nabla^T)$ denotes the symmetric gradient.
        In turn, for $\bF_\bm^d$ we have
        \begin{equation*}
            0 = \delta_{\bF^d_\bm} \sR_\varepsilon = \mu_0 \chi^{-1} \bm - \mu_0 \bh + \tau \mu_0 \chi^{-1} \bF^d_\bm \quad \implies \quad \bF^d_\bm = - \tau^{-1} (\bm - \chi \bh).
        \end{equation*}
        Altogether, we obtain the dissipative quantities:
        \begin{equation} \label{eq:dissip-quant}
            \bF^d_\bu = \nu D \bu, \qquad \bF^d_\bm = -\tau^{-1} (\bm - \chi \bh).
        \end{equation}
        We now determine the reversible quantities $\bF^r_\bu$, and $\bF_p^r$. We first recall the rate of change of the energy from \eqref{eq:dt-cE-eps}:
        \begin{equation} \label{eq:cE}
            \begin{split}
                \frac{\d}{\dt} \sE_\varepsilon & = \big(\bF_\bu^r - \conv \bu + \div \bF_\bu^d, \bu \big) + \big(\div \bF_p^r - \varepsilon \bu \cdot \nabla p, p \big) \\
                & \quad + \mu_0 \chi^{-1} \big( \bF_\bm^d - (\bu \cdot \nabla) \bm, \bm \big) + \mu_0 \big(- \bF_\bm^d + \Conv{\bu}{\bm}, \bh \big) + \mu_0 \big(\pa_t \bh_a, \bh \big).
            \end{split}
        \end{equation}
        We note that, since $\bF_\bu^d = \nu D \bu$ and $\bF_\bm^d = - \tau^{-1} (\bm - \chi \bh)$ from \eqref{eq:dissip-quant}, we infer that terms 
        \begin{equation*} \label{eq:dissip-terms}
            \big(\div \bF^d_\bu, \bu \big) = -\nu \|D \bu\|^2, \quad \frac{\mu_0}{\chi} \big(\bF^d_m, \bm \big) - \mu_0 \big(\bF^d_m, \bh \big) = -\frac{\mu_0}{\tau \chi} \|\bm - \chi \bh\|^2
        \end{equation*}
        are dissipative. Moreover, guided by the incompressibility condition on $\bu$, we choose $\bF_p^r = - \bu$. Thus, for the reversible term $\bF_\bu^r$ not to affect the rate of change $\frac{\d}{\dt} \sE_\varepsilon$ of energy $\sE_\varepsilon$ given in \eqref{eq:cE}, we require that
        \begin{equation} \label{eq:bF_bu^r}
            \big(\bF_\bu^r - \conv \bu, \bu \big) + \big(- \div \bu - \varepsilon \bu \cdot \nabla p, p \big) - \mu_0 \chi^{-1} \big(\Conv{\bu}{\bm}, \bm - \chi \bh \big) = 0.
        \end{equation}
        To get an expression for $\bF_\bu^r$, the goal now is to rewrite the terms in \eqref{eq:bF_bu^r} so that the test function is $\bu$. Of course, there could be many ways to do so, but we shall proceed with one such way. 
        
        On the one hand, since $\int_\Omega (\bu \cdot \nabla) \bu \cdot \bu \dx = - \int_\Omega (\bu \cdot \nabla) \bu \cdot \bu \dx - \int_\Omega |\bu|^2 \div \bu \dx$ by integration by parts, the first term in \eqref{eq:bF_bu^r} can be rewritten as $\big(\bF_\bu^r + \frac12 \bu \div \bu, \bu \big)$.
        On the other hand, the second term in \eqref{eq:bF_bu^r} can be rewritten as follows:
        \[
        \big(- \div \bu - \varepsilon \bu \cdot \nabla p, p \big) = \big(\bu, \nabla p \big) + \varepsilon/2 \big(\div \bu, p^2 \big) = \big(\nabla [p - \varepsilon p^2/2], \bu \big).
        \]
        Finally, the third term in \eqref{eq:bF_bu^r} simplifies to
        \[
        \big(\Conv{\bu}{\bm}, \bm \big) = -\frac{1}{2} \big(|\bm|^2, \div \bu \big) = \frac12 \left( \nabla |\bm|^2 , \bu \right),
        \]
        and
        % \[
        \begin{multline}
            \big(\Conv{\bu}{\bm}, \bh \big) = - \big((\bu \cdot \nabla) \bh, \bm \big) - \big(\bm \cdot \bh, \div \bu \big)\\
            = - \big([\nabla^T \bh] \bm, \bu \big) + \big(\nabla(\bm \cdot \bh), \bu \big)
            = - \big([\nabla^T \bh] \bm, \bu \big) + \chi^{-1} \big(\nabla(\bm \cdot \chi \bh), \bu \big).
        \end{multline}
        % \]multline
        Altogether, we deduce that \eqref{eq:bF_bu^r} can be equivalently expressed as
        \[
        \big(\bF_\bu^r + \bu \div \bu/2 + \nabla \wt p - \mu_0 [\nabla^T \bh] \bm, \bu \big) = 0,
        \]
        where the modified pressure $\wt p$ is given by
        \begin{equation} \label{eq:modified-pressure}
            \widetilde p := p  - \varepsilon \frac{p^2}{2} + \frac{\mu_0}{\chi} \bm \cdot \left( \chi \bh - \frac{\bm}{2} \right).
        \end{equation}
        Therefore, we choose $\bF_\bu^r$ to satisfy
        % \[
        $
        \bF_\bu^r = - \frac12 \bu \div \bu - \nabla \wt p + \mu_0 [\nabla^T \bh] \bm.
        $
        % \]

    \subsection{Equations and formal limits}
        Plugging $\bF_\bu^d = \nu D \bu$, $\bF_\bu^r = - \frac12 \bu \div \bu - \nabla \wt p + \mu_0 [\nabla^T \bh] \bm$, $\bF_p^r = -\bu$, and $\bF_\bm^d = - \frac{1}{\tau} (\bm - \chi\bh)$ back into the general equations \eqref{eq:ff-gen} yields
        \begin{equation*}
            \begin{cases}
                \D_t \bu + \frac12 \bu \div \bu - \nu \div D \bu + \nabla \widetilde p = \mu_0 [\nabla^T \bh] \bm, \qquad \varepsilon \D_t p + \div \bu = 0, \\
                \D_t \bm = -\frac{1}{\tau} (\bm - \chi \bh), \qquad \bcurl \bh = \bzero, \qquad \bb = \mu_0(\bm + \bh), \qquad \div \bb = 0,
            \end{cases}
        \end{equation*}
        where we recall that the modified pressure $\wt p$ is given by \eqref{eq:modified-pressure}.
        Furthermore, since $\bcurl \bh = 0$ implies that $\bh = \nabla \phi$ for some potential scalar function $\phi$, we have that $\nabla \bh = \nabla \nabla \phi$ is symmetric. Consequently, the forcing term in the momentum equation $(\nabla^T \bh) \bm = (\nabla \bh) \bm = (\bm \cdot \nabla) \bh$ is the aforementioned Kelvin force (see section \ref{sec:intro}).
        Hence, the ferrofluid equations reduce to
        \begin{equation} \label{eq:ff-eps}
            \begin{cases}
                \D_t \bu + \frac12 \bu \div \bu - \nu \div D \bu + \nabla \widetilde p = \mu_0 (\bm \cdot \nabla) \bh, \qquad \varepsilon \D_t p + \div \bu = 0, \\
                \D_t \bm = -\frac{1}{\tau} (\bm - \chi \bh), \quad \bcurl \bh = 0, \quad \bb = \mu_0(\bm + \bh), \quad \div \bb = 0.
            \end{cases}
        \end{equation}
        Recall that reversible terms $\bF_\bu^r, \bF_p^r$ do not contribute to $\frac{\d}{\dt} \sE_\varepsilon$. 
        Therefore, we deduce from \eqref{eq:abst-en-law} that system \eqref{eq:ff-eps} satisfies the energy law
        \begin{equation} \label{eq:ff-en-eps}
            \frac{\d}{\dt} \sE_\varepsilon[\bu, \bm, \bh] + 2 \sD[\bu, \bm, \bh] = \mu_0 (\pa_t \bh_a, \bh).
        \end{equation}
        % \begin{equation} \label{eq:ff-en-eps}
        %     \frac12 \frac{\d}{\dt} \left[ \|\bu\|^2 + \frac{\mu_0}{\chi} \|\bm\|^2 + \mu_0 \|\bh\|^2 + \varepsilon \|p\|^2 \right] + \nu \|D \bu\|^2 + \frac{\mu_0}{\tau \chi} \|\bm - \chi \bh\|^2 = \mu_0 (\pa_t \bh_a, \bh).
        % \end{equation}
        The term $\frac12 \bu \div \bu$ in \eqref{eq:ff-eps} is often used in numerical schemes to skew-symmetrize the convective term when the discrete velocities are not exactly divergence-free. It appeared naturally in our derivation, and is crucial in achieving energy law \eqref{eq:ff-en-eps} given that $\div \bu = - \varepsilon \D_t p \neq 0$.
        \begin{proposition}[governing PDEs and energy law]
            The ferrofluid system is governed by the following PDEs in the variables $(\bu, \wt p, \bh, \bm)$
            \begin{equation} \label{eq:ff-onsager}
                \begin{cases}
                    \D_t \bu - \nu \Delta \bu + \nabla \widetilde p = \mu_0 (\bm \cdot \nabla) \bh, \qquad \D_t \bm = -\frac{1}{\tau} (\bm - \chi \bh), \\
                    \div \bu = 0, \qquad \bcurl \bh = 0, \qquad \div (\bm + \bh) = 0,
                \end{cases}
            \end{equation}
            where the modified pressure $\wt p$ is given by $\wt p = p + \frac{\mu_0}{\chi} \bm \cdot \left( \chi \bh - \frac{\bm}{2} \right)$. Moreover, with $\sE := \sE_0$ and $\sD$ defined in \eqref{eq:energies}, i.e., 
            \begin{equation} \label{eq:sE-sD}
                \begin{split}
                    \sE & = \frac{1}{2} \|\bu\|^2 + \frac{\mu_0}{2 \chi} \|\bm\|^2 + \frac{\mu_0}{2} \|\bh\|^2, \quad \sD = \frac{1}{2 \nu} \|\bF^d_\bu\|^2 + \frac{\tau \mu_0}{2 \chi} \|\bF^d_\bm\|^2,
                \end{split}
            \end{equation}
            this system obeys the energy law
            \begin{equation} \label{eq:en-law-short-not}
                \frac{\d}{\dt} \sE[\bu, \bm, \bh] + 2 \sD[\bu, \bm, \bh] = \mu_0 \big( \pa_t \bh_a, \bh \big).
            \end{equation}
        \end{proposition}
        \begin{proof}
            We take the formal limit as $\varepsilon \downarrow 0$ in \eqref{eq:ff-eps}, \eqref{eq:ff-en-eps} and note that $\div \bu = 0$, which implies $\div D \bu = \Delta \bu$. We also observe that $\div (\bm + \bh) = \mu_0^{-1} \div \bb = 0$.
        \end{proof}

    \subsection{Making the boundary data homogeneous} \label{sec:energy-law}
        Homogenizing the boundary condition $(\bm + \bh) \cdot \bn_{\pa \Omega} = \bh_a \cdot \bn_{\pa \Omega}$ will prove useful for the purposes of discretization. To that end, we define the shifted variable $\wt \bh := \bh - \bh_a$ that satisfies $(\bm + \wt \bh) \cdot \bn_{\pa \Omega} = 0$, and note that replacing $\bh$ with $\wt \bh + \bh_a$ in \eqref{eq:ff-onsager} yields
        \begin{equation} \label{eq:ff-shifted}
            \begin{cases}
                \D_t \bu - \nu \Delta \bu + \nabla p = \mu_0 (\bm \cdot \nabla) (\wt \bh + \bh_a), \qquad \D_t \bm = -\frac{1}{\tau} (\bm - \chi [\wt \bh + \bh_a]), \\
                \div \bu = 0, \qquad \bcurl \wt \bh = 0, \qquad \div (\bm + \wt \bh) = 0,
            \end{cases}
        \end{equation}
        where we relabeled $p$ to mean $\wt p$, and the last two equations remain the same for $\wt \bh$ because $\bh_a$ is harmonic. We next derive an alternative energy law for the shifted variable $\wt \bh$.
        \begin{lemma}[energy law for shifted $\bh$] \label{lem:cont-energy-law}
            The solution $(\bu, \bm, \bh)$ of \eqref{eq:ff-onsager} satisfies the following energy law for $\wt \bh = \bh - \bh_a$:
            \begin{equation} \label{eq:ff-en-law}
                \frac{\d}{\dt} \sE[\bu, \bm, \wt \bh] + 2 \sD[\bu, \bm, \wt \bh] = \mu_0 \int_\Omega (\bm \cdot \nabla) \bh_a \cdot \bu \dx + \frac{\mu_0}{\tau} \big(\bm - \chi \wt \bh, \bh_a \big).
            \end{equation}
        \end{lemma}
        \begin{proof}
            We insert $\bh = \wt \bh + \bh_a$ into \eqref{eq:en-law-short-not} to arrive at
            \begin{equation} \label{eq:ff-en-law-wt-bh}
                \frac{\d}{\dt} \sE[\bu, \bm, \wt \bh + \bh_a] + 2 \sD[\bu, \bm, \wt \bh + \bh_a] = \mu_0 \big( \pa_t \bh_a, \wt \bh + \bh_a \big).
            \end{equation}
            Let $\bbI = \frac{\d}{\dt} \sE[\bu, \bm, \wt \bh] + 2 \sD[\bu, \bm, \wt \bh]$ be the quantity of interest. Expanding the squares containing $\wt \bh + \bh_a$ in \eqref{eq:ff-en-law-wt-bh}, we arrive at
            \begin{equation} \label{eq:bbI}
                % \begin{split}
                    \bbI = \mu_0 \big( \pa_t \bh_a, \wt \bh + \bh_a \big) - \mu_0 \frac{\d}{\dt} \big(\wt \bh, \bh_a \big) - \frac{\mu_0}{2} \frac{\d}{\dt} \|\bh_a\|^2 + \frac{2 \mu_0}{\tau} \big(\bm - \chi \wt \bh, \bh_a \big) - \frac{\mu_0 \chi}{\tau} \|\bh_a\|^2.
                % \end{split}
            \end{equation}
            By direct calculation, the terms involving derivatives in time read
            \[
            \begin{split}
                \mu_0 \big( \pa_t \bh_a, \wt \bh + \bh_a \big) - \mu_0 \frac{\d}{\dt} \big(\wt \bh, \bh_a \big) - \frac{\mu_0}{2} \frac{\d}{\dt} \|\bh_a\|^2 = - \mu_0 \big(\pa_t \wt \bh, \bh_a \big).
            \end{split}
            \]
            Using $\bh = \mu_0^{-1} \bb - \bm$ and $\bh = \wt \bh + \bh_a$, we deduce that $\pa_t \wt \bh = \pa_t [\mu_0^{-1} \bb - \bh_a] - \pa_t \bm$. Moreover, since $\bh_a = \nabla \phi_a$ is harmonic, we find that $\big( \pa_t [\mu_0^{-1} \bb - \bh_a], \bh_a \big) = 0$ from \eqref{eq:pat-bb}, and deduce that
            \[
            \begin{split}
                \big(\pa_t \wt \bh, \bh_a \big) & = \big( \pa_t [\mu_0^{-1} \bb - \bh_a], \bh_a \big) - \big(\pa_t \bm, \bh_a \big) \\
                & = - \big(\pa_t \bm, \bh_a \big) = \tau^{-1} \big(\bm - \chi(\wt \bh + \bh_a), \bh_a \big) + \big([\bu \cdot \nabla] \bm, \bh_a \big).
            \end{split}
            \]            
            Therefore, going back to \eqref{eq:bbI}, after simplifications, we obtain
            \begin{equation} \label{eq:bbI-end}
                \begin{split}
                    \bbI & = \mu_0 \tau^{-1} \big(\bm - \chi \wt \bh, \bh_a \big) - \mu_0 \big([\bu \cdot \nabla] \bm, \bh_a \big).
                \end{split}
            \end{equation}
            Finally, in view of $\div \bu = 0$ and $\nabla \bh_a = \nabla \nabla \phi_a = \nabla^T \bh_a$ because $\bh_a = \nabla \phi_a$ for some $\phi_a$, we deduce that
            \[
            \big([\bu \cdot \nabla] \bm, \bh_a \big) = - \big([\bu \cdot \nabla] \bh_a, \bm \big) = \big([\bm \cdot \nabla] \bh_a, \bu \big),
            \]
            which, combined with \eqref{eq:bbI-end}, yields the desired result.
        \end{proof}
        We note that the energy law \eqref{eq:ff-en-law} can also be obtained by testing the system \eqref{eq:ff-onsager} with suitable functions, as we will see in section \ref{sec:discr-en-law}.
    
\section{Space discretization} \label{sec:discretization}
    In this section, we propose a structure-preserving Galerkin scheme that captures all crucial properties satisfied by the exact solutions of system \eqref{eq:ff-shifted}:
    \begin{itemize}
        \item The discrete velocity $\bU$ is exactly divergence-free. 
        \item The discrete H-field $\bH$ is exactly curl-free and its gradient $\nabla \bH$ is a.e.~symmetric. 
        \item The discrete B-field $\mu_0(\bM + \bH)$ is exactly divergence-free.
    \end{itemize}
    These properties appear to be crucial to derive error estimates based on the discrete version of the energy law \eqref{eq:ff-en-law}. The main difficulty, again, is the fact that the energy law only supplies control of $\bm$ and $\bh$ in $L^\infty(I; \bL^2(\Omega))$, while the system \eqref{eq:ff-shifted} has derivatives of these variables in it (in particular, the problematic Kelvin force $\mu_0 (\bm \cdot \nabla) \bh$). Alluding to the introduction, we reiterate that this issue cannot be solved by simply adding a regularization term $\sigma \Delta \bm$ to the magnetization equation. Attempting to close the error estimate without relying on additional spatial control of discrete $\bm, \bh$ (i.e., with $\sigma = 0$) led us to the three requirements above, which appear to be new.
    
    These properties are ensured by our method thanks to the introduction of the scalar potentials $\phi, \psi$ in section \ref{sec:scal-pot}.
    As a result, we prove that the discretization error is controlled by a projection error.
    Finally, we provide explicit examples of such schemes and derive their corresponding error estimates.

    \subsection{Reformulation using two scalar potentials} \label{sec:scal-pot}
        We assume that $d = 2$, and from now on we shall \textit{relabel $\bh$ to mean $\wt \bh$} from section \ref{sec:energy-law}. Since $\div (\bm + \bh) = 0$ and $\Omega$ is simply connected, we deduce the existence of a scalar potential $\psi$ that satisfies
        \begin{equation*}
            \bm + \bh = \bcurl \psi \text{ in } \Omega, \qquad \bcurl \psi \cdot \bn_{\pa \Omega} = 0 \text{ on } \pa \Omega, \qquad \big(\psi, 1 \big) = 0.
        \end{equation*}
        Similarly, because of $\curl \bh = \mathbf{0}$, there is a scalar potential $\phi$ that satisfies
        \begin{equation*}
            \bh = \nabla \phi \text{ in } \Omega, \qquad \big(\phi, 1 \big) = 0.
        \end{equation*}
        With this in mind, we can reformulate the ferrofluid system \eqref{eq:ff-shifted} as follows: given initial data $\bu_0, \psi_0, \phi_0$ and forcing data $\bh_a$, find $(\bu, p, \psi, \phi)$ with $\int_\Omega p \dx = \int_\Omega \phi \dx = \int_\Omega \psi \dx = 0$, such that
        \begin{equation} \label{eq:ff-hom}
            \begin{cases}
                \D_t \bu - \nu \Delta \bu + \nabla p = \mu_0 (\bm \cdot \nabla) (\bh + \bh_a), \qquad \D_t \bm = - \frac{1}{\tau} (\bm - \chi (\bh + \bh_a)), \\
                \div \bu = 0, \qquad \bh = \nabla \phi, \qquad \bm = \bcurl \psi - \nabla \phi,
            \end{cases}
        \end{equation}
        supplemented with boundary and initial conditions
        \begin{subequations}
            \begin{equation} \label{eq:ff-hom-bc-init}
                \bu|_{\pa \Omega} = \mathbf{0}, \qquad (\bcurl \psi \cdot \bn_{\pa \Omega})|_{\pa \Omega} = 0,
            \end{equation}
            \vspace{-2em}
            \begin{equation} \label{eq:ff-hom-ic}
                \bu(0) = \bu_0, \qquad \bm(0) = \bcurl \psi_0 - \nabla \phi_0, \qquad \bh(0) = \nabla \phi_0.
            \end{equation}
        \end{subequations}
        The boundary condition $\bcurl \psi \cdot \bn_{\pa \Omega} = 0$ equivalently reads as
        \[
        \nabla_{\pa \Omega} \psi = 0, \quad \text{ on } \pa \Omega,
        \]
        where $\nabla_{\pa \Omega}$ is the tangential derivative to $\pa \Omega$.
        This is equivalent to $\psi$ being constant on $\pa \Omega$ because $\pa \Omega$ is connected.
        Since the precise value of the constant is irrelevant, we set $\psi = 0$ on $\pa \Omega$ and drop the uniqueness condition $\int_\Omega \psi \dx = 0$.

        The simplified system now reads: given data $(\bu_0, \psi_0, \phi_0, \bh_a)$, find $(\bu, p, \psi, \phi)$ with $\int_\Omega p \dx = \int_\Omega \phi \dx = 0$ that solves \eqref{eq:ff-hom}, \eqref{eq:ff-hom-ic} and instead of \eqref{eq:ff-hom-bc-init} satisfies the homogeneous boundary conditions
        \begin{equation*}
            \bu|_{\pa \Omega} = \mathbf{0}, \qquad \psi|_{\pa \Omega} = 0.
        \end{equation*}
        It is now clear that the assumption that $d=2$ is critical to be able to use a \textit{scalar} potential for $\bb$ that satisfies the \textit{homogeneous} Dirichlet boundary condition.
        Extending the forthcoming analysis to $d=3$ is possible, but it requires non-trivial modifications that are carried out in \cite{BenavidesNochetto2026-ferrofluids}.
    
    \subsection{Spaces}\label{sec:spaces}
        We denote the finite dimensional subspaces for $\bu, p$, and $\psi, \phi$ respectively as
        \begin{equation}
            \bbU \subseteq \bH^1_0(\Omega), \quad \bbP \subseteq L^2_\#(\Omega), \quad \bbX \subseteq H^1(\Omega),
        \end{equation}
        where $L^2_\#(\Omega) := \{q \in L^2(\Omega): \int_\Omega q = 0\}$. We assume that the pair $(\bbU, \bbP)$ satisfies the uniform discrete inf-sup condition
        \begin{equation}\label{eq:disc-inf-sup}
        \inf_{\bV \in \bbU : \|\nabla \bV\| = 1} \sup_{Q \in \bbP : \|Q\| = 1} (\div \bV, Q)  \geq \beta > 0,
        \end{equation}
        where $\beta$ is independent of the cardinalities of $(\bbU, \bbP)$. We introduce the following subspaces of $\bbX$ to be used for $\phi, \psi$:
        \begin{equation}\label{eq:Xsubspaces}
            \bbX_\# := \bbX \cap L^2_\#(\Omega), \qquad \bbX_0 := \bbX \cap H^1_0(\Omega).
        \end{equation}
        Even though these mild assumptions are very convenient for discretization, they are insufficient for our perturbation argument to hold true.
        We therefore make the following stronger assumptions on the spaces $\bbU, \bbP, \bbX_\#, \bbX_0$.
        \begin{assump}[Galerkin scheme] \label{assump:scheme}
            We assume that $(\bbU, \bbP, \bbX)$ additionally satisfy
            \begin{enumerate}[label=\textbf{A\arabic*}]
                \item ($H^2$-conformity of $\Phi, \Psi$): ~ $\bbX \subseteq H^2(\Omega)$.
                \item (exactly $\div$-free $\bU$): ~ If $(\div \bU, Q) = 0$ for all $Q \in \bbP$, then $\div \bU = 0$ a.e. in $\Omega$.
            \end{enumerate}
        \end{assump}
        \textit{\textbf{A1}} implies \textit{pointwise} symmetry of $\nabla \bH$, where $\bH = \nabla \Phi$ is the discrete $H$-field, while \textit{\textbf{A2}} leads to the skew-symmetry of the \textit{unperturbed} convective trilinear form.
        More precise implications of these two assumptions will become clearer in the next subsection. We mention that enforcing only the Dirichlet boundary condition on an $H^2$-conforming finite element space is not a trivial matter, but it can be achieved via an Uzawa--Powell--Nesterov interation; see section \ref{sec:experiment}.
        
    \subsection{Semidiscrete method}
        Let $\cB : \bH^1(\Omega) \times \bH^1(\Omega) \times \bH^1(\Omega) \to \bbR$ be defined as
        \begin{equation*}
            \cB(\bu, \bv, \bw) := \int_\Omega u_i \pa_i v_j w_j \dx = \int_\Omega (\bu \cdot \nabla) \bv \cdot \bw \dx. 
        \end{equation*}
        Given discrete initial data $(\bU_0, \Phi_0, \Psi_0) \in \bbU \times \bbX_\# \times \bbX_0$, we set
        \[
        \bU(0) = \bU_0, \qquad \bH(0) = \nabla \Phi_0, \qquad \bM(0) = \nabla \Phi_0 + \bcurl \Psi_0.
        \]
        Given forcing $\bh_a$, for a.e. $t \in (0,T)$, we seek $(\bU, P, \Phi, \Psi) \in \bbU \times \bbP \times \bbX_\# \times \bbX_0$ such that
        \begin{equation} \label{eq:semidiscrete}
            \begin{cases}
                \big(\D_t \bU, \bV_\bu \big) + \nu \big(\nabla \bU, \nabla \bV_\bu \big) - \big(P, \div \bV_\bu \big) & = \mu_0 \cB(\bM, \bH, \bV_\bu) + \mu_0 \cB(\bM, \bh_a, \bV_\bu), \\
                (\div \bU, V_p) & = 0, \\
                \big(\D_t \bM, \bV_\bm \big) & = - \frac{1}{\tau} \big(\bM - \chi \bH, \bV_\bm \big) + \frac{\chi}{\tau} \big(\bh_a, \bV_\bm \big),
            \end{cases}
        \end{equation}
        for all $(\bV_\bu, V_p, \bV_\bm) \in \bbU \times \bbP \times [\nabla \bbX_\# \oplus \bcurl \bbX_0]$. We stress that, from now on, to avoid complicating the notation we redefine $\D_t$ to be the material derivative with respect to the \textit{discrete} velocity, i.e., $\D_t = \pa_t + (\bU \cdot \nabla)$, and
        \begin{equation} \label{eq:bH-bM}
            \bH := \nabla \Phi \in \bH^1(\Omega), \qquad \bM := \bcurl \Psi - \nabla \Phi \in \nabla \bbX_\# \oplus \bcurl \bbX_0 \subseteq \bH^1(\Omega).
        \end{equation}
        Note that because the test space for the second equation is a sum space, the three equations in our semidiscrete method can also be regarded as four equations.

        Since we have already committed to an $H^2$-conforming space $\bbX$, we mention that it is possible to use the stream function formulation for the velocity.
        This would eliminate $P$ and reduce the velocity to one scalar variable at the expense of having to deal with a fourth-order problem which we expect to be worse-conditioned than \eqref{eq:semidiscrete}.
        
        Let us now elucidate the key properties of solutions $\sfU := (\bU, P, \bM, \bH)$ of \eqref{eq:semidiscrete}. From (\textbf{A2}) and \eqref{eq:semidiscrete}, we know that $\div \bU = 0$, and from \eqref{eq:bH-bM} we know that $\div (\bM + \bH) = \div \bcurl \Psi = 0$, hence
        \begin{subequations}
            \begin{equation} \label{eq:skew-u} \tag{skew-$\bU$}
                \cB(\bU, \bv, \bw) = - \cB(\bU, \bw, \bv), \qquad
            \end{equation}
            \vspace{-2em}
            \begin{equation} \label{eq:skew-m+h} \tag{skew-$\bM+\bH$}
                \qquad \ \ \ \quad \cB(\bM + \bH, \bv, \bw) = - \cB(\bM + \bH, \bw, \bv),
            \end{equation}
        \end{subequations}
        for all $\bw, \bv \in \bH^1(\Omega)$. The boundary terms vanish because $\bU = 0$ and $\Psi = 0$ on $\pa \Omega$, whence $(\bM + \bH) \cdot \bn_{\pa \Omega} = \bcurl \Psi \cdot \bn_{\pa \Omega} = \nabla_{\pa \Omega} \Psi = \bzero$. From (\textbf{A1}) and \eqref{eq:bH-bM}, we know that $\nabla \bH = \nabla \nabla \Phi = \nabla^T \bH$ is symmetric, whence
        \begin{equation} \label{eq:sym} \tag{symm-$\nabla \bH$}
            \cB(\bu, \bH, \bw) = \cB(\bw, \bH, \bu), \qquad \forall \bu, \bw \in \bH^1(\Omega).
        \end{equation}
        These precise identities at the discrete level are not required to derive an energy law, but rather to carry out our error analysis.
        
        The semidiscrete problem \eqref{eq:semidiscrete} is well-posed on $(0, T^*)$ for some $T^* > 0$ due to the Cauchy--Lipschitz theorem for ODEs. Obtaining a priori bounds in the next subsection allows extending the interval of well-posedness all the way to $T^* = T$.
        
    \subsection{Discrete energy law} \label{sec:discr-en-law}
        The following discrete energy law mimics the continuous energy law of section \ref{sec:energy-law}.
        \begin{lemma}[discrete energy law] \label{lem:discr-energy}
            The semidiscrete solution $(\bU, P, \Phi, \Psi)$ of \eqref{eq:semidiscrete} satisfies
            \begin{equation} \label{eq:discr-en-eq}
                \frac{\d}{\dt} \sE[\bU, \bM, \bH](t) + 2 \sD[\bU, \bM, \bH](t) = \mu_0 \cB(\bM, \bh_a, \bU) + \frac{\mu_0}{\tau} \big(\bh_a, \bM - \chi \bH \big).
            \end{equation}
        \end{lemma}
        \begin{proof}
            We test \eqref{eq:semidiscrete} with $\bV_\bu = \bU$, $\bV_\bm = \frac{\mu_0}{\chi} (\bM - \chi \bH)$, and note that $\cB(\bU, \bU, \bU) = 0$ due to \eqref{eq:skew-u}. This leads to
            \begin{equation} \label{eq:en-law-pf-1}
                \begin{split}
                    \frac12 \frac{\d}{\dt} \|\bU\|^2 + \nu \|\nabla \bU\|^2 & = \mu_0 \cB(\bM, \bH, \bU) + \mu_0 \cB(\bM, \bh_a, \bU), \\
                    \frac{\mu_0}{2 \chi} \frac{\d}{\dt} \|\bM\|^2 - \mu_0 \big(\pa_t \bM, \bH \big) - \mu_0 \cB(\bU, \bM, \bH) & = - \frac{\mu_0}{\tau \chi} \|\bM - \chi \bH\|^2 + \frac{\mu_0}{\tau} \big(\bh_a, \bM - \chi \bH \big).
                \end{split}
            \end{equation}
            We next crucially use \eqref{eq:sym} and \eqref{eq:skew-u} to arrive at
            \[
            \cB(\bM, \bH, \bU) = \cB(\bU, \bH, \bM) = - \cB(\bU, \bM, \bH).
            \]
            This allows canceling out the Kelvin force term $\mu_0 \cB(\bM, \bH, \bU)$ with the convective term $-\mu_0 \cB(\bU, \bM, \bH)$ in \eqref{eq:en-law-pf-1}. Thus, adding the two equalities in \eqref{eq:en-law-pf-1} yields
            \[
            \begin{split}
                \frac12 \frac{\d}{\dt} \|\bU\|^2 & + \frac{\mu_0}{2 \chi} \frac{\d}{\dt} \|\bM\|^2 + \nu \|\nabla \bU\|^2 + \frac{\mu_0}{\tau \chi} \|\bM - \chi \bH\|^2 \\
                & = \mu_0 \big(\pa_t \bM, \bH \big) + \mu_0 \cB(\bM, \bh_a, \bU) + \frac{\mu_0}{\tau} \big(\bh_a, \bM - \chi \bH \big).
            \end{split}
            \]
            Using the fact that $\bM = \bcurl \Psi - \nabla \Phi$ and that $\bH = \nabla \Phi$, we get
            \[
            \big(\pa_t \bM, \bH \big) = \big(\pa_t [\bcurl \Psi - \nabla \Phi], \nabla \Phi \big) = \big(\bcurl \pa_t \Psi - \nabla \pa_t \Phi, \nabla \Phi \big) = - \frac12 \frac{\d}{\dt} \|\bH\|^2,
            \]
            because $(\bcurl \pa_t \Psi, \nabla \Phi) = 0$. Altogether, we get the asserted energy equality.
        \end{proof}
        Let $I := (0,T)$ be a time interval.
        We next deduce an energy estimate for the discrete energy law \eqref{eq:discr-en-eq}.
        \begin{lemma}[energy estimate]
            We have the following energy estimate for all $t \in (0,T]$,
            \[
            \begin{split}
                \sE[\bU, \bM, \bH](t) + \int_0^t \sD[\bU, \bM, \bH](s) \ds & \leq \left(\sE[\bU, \bM, \bH](0) + \frac{\mu_0 \chi}{2 \tau} \int_I \|\bh_a\|^2 \ds \right) \\
                & \quad \times \exp\left(\mu_0^{1/2} \chi^{1/2} \|\nabla \bh_a\|_{L^1(I; \dL^\infty)} \right).
            \end{split}
            \]
        \end{lemma}
        \begin{proof}
            We start with the discrete energy equality \eqref{eq:discr-en-eq}:
            \begin{equation} \label{eq:discr-en-eq-2}
                \frac{\d}{\dt} \sE[\bU, \bM, \bH](t) + 2 \sD[\bU, \bM, \bH](t) = \mu_0 \cB(\bM, \bh_a, \bU) + \frac{\mu_0}{\tau} \big(\bh_a, \bM - \chi \bH \big).
            \end{equation}
            We need to estimate the two terms on the right-hand side above. We have
            \[
            \frac{\mu_0}{\tau} |\big(\bh_a, \bM - \chi \bH \big)| \leq \frac{\chi \mu_0}{2 \tau} \|\bh_a\|^2 + \frac{\mu_0}{2 \chi \tau} \|\bM - \chi \bH\|^2,
            \]
            from which the second term above can be absorbed on the left-hand side of \eqref{eq:discr-en-eq-2}. For the other, nonlinear term, we factor out $\|\nabla \bh_a\|_{\dL^\infty}$ and use Young's inequality with $\varepsilon = \frac{\mu_0^{1/2}}{\chi^{1/2}}$ to arrive at
            \[
            \begin{split}
                \mu_0 |\cB(\bM, \bh_a, \bU)| & \leq \mu_0^{1/2} \chi^{1/2} \|\nabla \bh_a\|_{\dL^\infty} \left(\frac{\mu_0}{2 \chi} \|\bM\|^2 + \frac{1}{2} \|\bU\|^2 \right).
            \end{split}
            \]
            Integrating \eqref{eq:discr-en-eq} in time yields
            \[
            \begin{split}
                \sE[\bU, \bM, \bH](t) & + \int_0^t \sD[\bU, \bM, \bH](s) \ds \leq \sE[\bU, \bM, \bH](0) + \frac{\mu_0 \chi}{2 \tau} \int_I \|\bh_a\|^2 \ds \\
                & \qquad + \mu_0^{1/2} \chi^{1/2} \|\nabla \bh_a\|_{L^1(I; \dL^\infty)} \int_0^t \sE[\bU, \bM, \bH](s) \ds.
            \end{split}
            \]
            Applying Gr\"onwall's inequality yields the desired estimate.
        \end{proof}
        
    \subsection{Discrete perturbation estimate} \label{sec:pert-est}
        Let $(\bu, p, \phi, \psi)$ be a smooth weak solution corresponding to smooth data $(\bu_0$, $\bm_0$, $\bh_0$, $\bh_a)$. We let $\osfU := (\obU, \oP, \oPhi, \oPsi)$ be suitable projections of $(\bu, p, \phi, \psi)$ onto $\bbU \times \bbP \times \bbX_\# \times \bbX_0$ with the requirement that $\div \obU = 0$. We define the Galerkin discretization, projection, and discrete errors, respectively, as follows:
        \[
        \begin{split}
            \sfe & := (\be_\bu, e_p, \be_\bm, \be_\bh) := (\bu - \bU, p - P, \bm - \bM, \bh - \bH), \\
            \sfe^p & := (\be_\bu^p, e_p^p, \be_\bm^p, \be_\bh^p) := (\bu - \obU, p - \oP, \bm - \obM, \bh - \obH), \\
            \sfE & := (\bE_\bu, E_p, \bE_\bm, \bE_\bh) := (\bU - \obU, P - \oP, \bM - \obM, \bH - \obH).
        \end{split}
        \]
        Before we proceed with the error equation, we shall introduce the following notation for the operators associated with the momentum and magnetization equations:
        \[
        \begin{split}
            \big(\sM_\bu \sfu, \bV_\bu \big) & := \big(\pa_t \bu, \bV_\bu \big) + \cB(\bu, \bu, \bV_\bu) + \nu \big(\nabla \bu, \nabla \bV_\bu \big) - \big(p, \div \bV_\bu \big), \\
            & \quad - \mu_0 \cB(\bm, \bh, \bV_\bu) - \mu_0 \cB(\bm, \bh_a, \bV_\bu), \\ 
            \big(\sM_\bm \sfu, \bV_\bm \big) & := \big(\pa_t \bm, \bV_\bm \big) + \cB(\bu, \bm, \bV_\bm) + \frac{1}{\tau} \big(\bm - \chi \bh, \bV_\bm \big),
        \end{split}
        \]
        where $\sfu = (\bu, p, \phi, \psi) \in \bH^1_0(\Omega) \times L^2_\#(\Omega) \times H^2(\Omega) \cap L^2_\#(\Omega) \times H^2(\Omega) \cap H^1_0(\Omega)$, and $\bh = \nabla \phi, \bm = \bcurl \psi - \nabla \phi$, and $(\bV_\bu, \bV_\bm) \in \bbU \times (\nabla \bbX_\# \oplus \bcurl \bbX_0)$. Since $\sfu$ is a smooth weak solution, we see that
        \begin{equation} \label{eq:weak-sol-iden}
            \big(\sM_\bu \sfu, \bV_\bu \big) = 0, \quad \forall \bV_\bu \in \bbU, \qquad \big(\sM_\bm \sfu, \bV_\bm \big) = \frac{\chi}{\tau} \big(\bh_a, \bV_\bm \big), \quad \forall \bV_\bm \in \nabla \bbX_\# \oplus \bcurl \bbX_0.
        \end{equation}
        Below we will apply $\sM_\bu$ and $\sM_\bm$ to $\sfE$ and quantify the deviation from zero of the right-hand side of \eqref{eq:weak-sol-iden} in terms of the projection error $\sfe^p$. Writing the error equations this way (i.e., using nonlinear operators $\sM_\bu, \sM_\bm$ from the Galerkin method) lets us apply the delicate discrete energy law \eqref{eq:discr-en-eq} directly to the discrete error $\sfE$. We now derive the error equation.
        \begin{lemma}[error equation]
            The discrete error $\sfE = (\bE_\bu, E_p, \bE_\bm, \bE_\bh)$ satisfies
            \begin{equation} \label{eq:err-eq}
                \begin{cases}
                    \big(\sM_\bu \sfE, \bV_\bu \big) = \big(\bR_\bu, \bV_\bu \big), \qquad \big(\sM_\bm \sfE, \bV_\bm \big) = \big(\bR_\bm, \bV_\bm \big), \\
                    \big(\div \bE_\bu, V_p \big) = 0, \qquad \bE_\bh = \nabla E_\phi, \qquad \bE_\bm = \bcurl E_\psi - \nabla E_\phi,
                \end{cases}
            \end{equation}
            for all $(\bV_\bu, V_p, \bV_\bm) \in \bbU \times \bbP \times (\nabla \bbX_\# \oplus \bcurl \bbX_0)$. The residuals $\bR_\bu, \bR_\bm$ are given by
            \begin{equation} \label{eq:residuals}
                % \begin{split}
                    \big(\bR_\bu, \bV_\bu \big) = \big(\bN_\bU, \bV_\bu \big) - \big(\sM_\bu \osfU, \bV_\bu \big), \qquad
                    \big(\bR_\bm, \bV_\bm \big) = \big(\bN_\bM, \bV_\bm \big) - \big(\sM_\bm \osfU, \bV_\bm \big),
                % \end{split}
            \end{equation}
            where $\bN_\bU, \bN_\bM$ account for the nonlinear structure of the system and are given by
            \begin{multline} \label{eq:bN}
                \big(\bN_\bU, \bV_\bu \big) := - \cB(\bE_\bu, \obU, \bV_\bu) - \cB(\obU, \bE_\bu, \bV_\bu) + \mu_0 \cB(\bE_\bm, \obH, \bV_\bu) + \mu_0 \cB(\obM, \bE_\bh, \bV_\bu), \\
                \big(\bN_\bM, \bV_\bm \big) := - \cB(\obU, \bE_\bm, \bV_\bm) - \cB(\bE_\bu, \obM, \bV_\bm). \qquad \qquad \qquad \qquad \qquad \quad \ \ \ \,
            \end{multline}
        \end{lemma}
        \begin{proof}
            We proceed in three steps, examining each equation separately. \smallskip

            \noindent \textbf{Step 1}: \textit{Momentum equation}. We define $\bR_\bu$ to be the result of inserting $\sfE$ into the momentum equation, namely $\big(\bR_\bu, \bV_\bu \big) := \big(\sM_\bu \sfE, \bV_\bu \big)$. We first recall that the Galerkin solution $\sfU$ solves the semidiscrete problem \eqref{eq:semidiscrete}: $(\sM_\bu \sfU, \bV_\bu) = 0$. We next observe that $\sM_\bu$ has a couple of nonlinear terms, the convective term and the Kelvin force. When evaluated on $\sfE$, they give rise to the nonlinear terms $(\bN_\bU, \bV_\bu)$. To see this, consider the contribution of the convective terms:
            \[
            \begin{split}
                \cB(\bE_\bu, \bE_\bu, \bV_\bu) - \cB(\bU, \bU, \bV_\bu) + \cB(\obU, \obU, \bV_\bu) = - \cB(\bE_\bu, \obU, \bV_\bu) - \cB(\obU, \bE_\bu, \bV_\bu).
            \end{split}
            \]
            Likewise, the Kelvin force yields
            \[
            \begin{split}
                -\mu_0[\cB(\bE_\bm, \bE_\bm, \bV_\bu) & - \cB(\bM, \bH, \bV_\bu) + \cB(\obM, \obH, \bV_\bu)] \\
                & = \mu_0[\cB(\bE_\bm, \obH, \bV_\bu) + \cB(\obM, \bE_\bh, \bV_\bu)].
            \end{split}
            \]
            Adding these two terms gives $(\bN_\bU, \bV_\bu)$ and leads to
            \[
            \big(\bR_\bu, \bV_\bu \big) = \big(\sM_\bu \sfE, \bV_\bu \big) = \big(\bN_\bU, \bV_\bu \big) - \big(\sM_\bu \osfU, \bV_\bu \big).
            \]
            \textbf{Step 2}: \textit{Magnetization equation}. We now define $\bR_\bm$ to be $\sM_\bm$ evaluated at $\sfE$, namely $(\bR_\bm, \bV_\bm) := (\sM_\bm \sfE, \bV_\bm)$. Proceeding as in Step 1, we obtain
            \[
            (\bR_\bm, \bV_\bm) = (\sM_\bm \sfE, \bV_\bm) = (\bN_\bM, \bV_\bm) - (\sM_\bm \osfU, \bV_\bm).
            \]
            \textbf{Step 3}: \textit{Constraints}. Since the constraints are linear, we easily find that
            \[
            (\div \bE_\bu, V_p) = 0, \quad \nabla E_\phi = \bE_\bh, \quad \bcurl E_\psi - \nabla E_\phi = \bE_\bm.
            \]
            This finishes the proof.
        \end{proof}
        \begin{theorem}[discrete perturbation estimate]
            The following bound on the Galerkin discretization error is valid:
            \begin{equation} \label{eq:err-est}
            \begin{split}
                \sE [\be_\bu, \be_\bm, \be_\bh](t) & + \int_0^t \sD[\be_\bu, \be_\bm, \be_\bh] \ds \leq 2 \sE[\be_\bu^p, \be_\bm^p, \be_\bh^p](t) \\
                & + 2 \int_0^t \sD[\be_\bu^p, \be_\bm^p, \be_\bh^p] \ds + 2 (\sE_0 + \sG) \exp(C \|\eta\|_{L^1(I)}),
            \end{split}
            \end{equation}
            where 
            \begin{equation} \label{eq:init-consist-errs}
                \begin{split}
                    \sE_0 & := \sE[\bE_\bu, \bE_\bm, \bE_\bh](0), \qquad \sG := \int_I \frac{1}{2 \nu} \|\sM_\bu \osfU\|_{\bbU'}^2 + \frac{\chi \mu_0}{2 \tau} \|\sM_\bm \osfU\|^2 \ds
                \end{split}
            \end{equation}
            are the initial and consistency errors and
            \[
            \eta := \max\left(\mu_0 \|\nabla \bh_a\|_{\dL^\infty}, (1 + \chi) \|\nabla \obU\|_{\dL^\infty}, \mu_0 \|\nabla \obH\|_{\dL^\infty}, (1 + \chi + \mu_0) \|\nabla \obM\|_{\dL^\infty} \right).
            \]
        \end{theorem}
        \begin{proof}
            We see from \eqref{eq:err-eq} that the equations $\sM_\bu \sfE$ and $\sM_\bm \sfE$ for the error $\sfE$ contain the right-hand sides $\bR_\bu$ and $\bR_\bm$. They, together with the nonlinear Kelvin force, form the right-hand side of the discrete energy law \eqref{eq:en-law-pf-1}, namely
            \[
            \begin{split}
                \frac{\d}{\dt} \sE[\bE_\bu, \bE_\bm, \bE_\bh] + 2 \sD[\bE_\bu, \bE_\bm, \bE_\bh] & = \mu_0 \cB(\bE_\bm, \bh_a, \bE_\bu) \\
                & \quad + \big(\bR_\bu, \bE_\bu \big) + \frac{\mu_0}{\tau} \big(\bR_\bm, \bE_\bm - \chi \bE_\bh \big).
            \end{split}
            \]
            According to the definition \eqref{eq:residuals}, we write
            \[
            \begin{split}
                \big(\bR_\bu, \bE_\bu \big) + \frac{\mu_0}{\tau} \big(\bR_\bm, \bE_\bm - \chi \bE_\bh \big) & = \big(\bN_\bU, \bE_\bu \big) + \frac{\mu_0}{\tau} \big(\bN_\bM, \bE_\bm - \chi \bE_\bh \big) \\
                & \quad - \big(\sM_\bu \osfU, \bE_\bu \big) - \frac{\mu_0}{\tau} \big(\sM_\bm \osfU, \bE_\bm - \chi \bE_\bh \big).
            \end{split}
            \]
            We now manipulate $\bR_\bu, \bR_\bm$ separately, trying to move all derivatives to the projections $\osfU$. This amounts to moving all terms with overlines to the second argument of $\cB(\cdot, \cdot, \cdot)$, which is the one being differentiated. The other two arguments do not have derivatives. \smallskip

            \noindent \textit{Rewriting $\big(\bN_\bU, \bE_\bu \big)$}. We first use \eqref{eq:skew-u} to deduce $\cB(\obU, \bE_\bu, \bE_\bu) = 0$. Then \eqref{eq:bN} implies
            \[
            \begin{split}
                \big(\bN_\bU, \bE_\bu \big) & = - \cB(\bE_\bu, \obU, \bE_\bu) + \mu_0 \cB(\bE_\bm, \obH, \bE_\bu) + \mu_0 \cB(\obM, \bE_\bh, \bE_\bu).
            \end{split}
            \]
            Using \eqref{eq:sym} and \eqref{eq:skew-u}, we rewrite the last term above as
            \[
            \cB(\obM, \bE_\bh, \bE_\bu) = \cB(\bE_\bu, \bE_\bh, \obM) = - \cB(\bE_\bu, \obM, \bE_\bh).
            \]
            Therefore,
            \[
            \big(\bN_\bU, \bE_\bu \big) = - \cB(\bE_\bu, \obU, \bE_\bu) + \mu_0 \cB(\bE_\bm, \obH, \bE_\bu) - \mu_0 \cB(\bE_\bu, \obM, \bE_\bh).
            \]
            \textit{Rewriting $\big(\bN_\bM, \bE_\bm - \chi \bE_\bh \big)$}. We use the definition of $\bN_\bM$ and expand terms:
            \[
            \begin{split}
                \big(\bN_\bM, \bE_\bm - \chi \bE_\bh \big) & = - \cB(\obU, \bE_\bm, \bE_\bm - \chi \bE_\bh) - \cB(\bE_\bu, \obM, \bE_\bm - \chi \bE_\bh) \\
                & = \chi \cB(\obU, \bE_\bm, \bE_\bh) - \cB(\bE_\bu, \obM, \bE_\bm) + \chi \cB(\bE_\bu, \obM, \bE_\bh),
            \end{split}
            \]
            because \eqref{eq:skew-u} implies $\cB(\obU, \bE_\bm, \bE_\bm) = 0$. We now focus our attention on $\cB(\obU, \bE_\bm, \bE_\bh)$. We first use \eqref{eq:skew-u} to add $\cB(\obU, \bE_\bh, \bE_\bh) = 0$, and then use \eqref{eq:skew-u} again to swap the second and third arguments as follows:
            \[
                \cB(\obU, \bE_\bm, \bE_\bh) = \cB(\obU, \bE_\bm + \bE_\bh, \bE_\bh) = - \cB(\obU, \bE_\bh, \bE_\bm + \bE_\bh).
            \]
            We next employ \eqref{eq:sym} followed by \eqref{eq:skew-m+h} to get
            \[
            \cB(\obU, \bE_\bm, \bE_\bh) = - \cB(\bE_\bm + \bE_\bh, \bE_\bh, \obU) = \cB(\bE_\bm + \bE_\bh, \obU, \bE_\bh).
            \]
            Therefore,
            \[
            \big(\bN_\bM, \bE_\bm - \chi \bE_\bh \big) = \chi \cB(\bE_\bm + \bE_\bh, \obU, \bE_\bh) - \cB(\bE_\bu, \obM, \bE_\bm) + \chi \cB(\bE_\bu, \obM, \bE_\bh).
            \]
            Altogether, we obtain
            \[
            \begin{split}
                \frac{\d}{\dt} \sE[\bE_\bu, \bE_\bm, \bE_\bh] & + 2 \sD[\bE_\bu, \bE_\bm, \bE_\bh] \\
                & = \mu_0 \cB(\bE_\bm, \bh_a, \bE_\bu) - \cB(\bE_\bu, \obU, \bE_\bu) + \mu_0 \cB(\bE_\bm, \obH, \bE_\bu) \\
                & \quad - \mu_0 \cB(\bE_\bu, \obM, \bE_\bh) + \chi \cB(\bE_\bm + \bE_\bh, \obU, \bE_\bh) - \cB(\bE_\bu, \obM, \bE_\bm) \\
                & \quad + \chi \cB(\bE_\bu, \obM, \bE_\bh) - \big(\sM_\bu \osfU, \bE_\bu \big) - \frac{\mu_0}{\tau} \big(\sM_\bm \osfU, \bE_\bm - \chi \bE_\bh \big).
            \end{split}
            \]
            We now estimate each term on the right-hand side. The terms involving the trilinear form $\cB$ will all be estimated as follows: $|\cB(\ba, \bb, \bc)| \leq \|\nabla \bb\|_{\dL^\infty} \|\ba\| \|\bc\|$.
            % \[
            % |\cB(\ba, \bb, \bc)| \leq \|\nabla \bb\|_{\dL^\infty} \|\ba\| \|\bc\|.
            % \]
            Consistency errors are estimated as
            \[
            \begin{split}
                |\big(\sM_\bu \osfU, \bE_\bu \big)| & \leq \frac{1}{2 \nu} \|\sM_\bu \osfU\|_{\bbU'}^2 + \frac{\nu}{2} \|\nabla \bE_\bu\|^2, \\
                \frac{\mu_0}{\tau} |\big(\sM_\bm \osfU, \bE_\bm - \chi \bE_\bh \big)| & \leq \frac{\chi \mu_0}{2 \tau} \|\sM_\bm \osfU\|^2 + \frac{\mu_0}{2 \chi \tau} \|\bE_\bm - \chi \bE_\bh\|^2.
            \end{split}
            \]
            Altogether, we end up with
            \[
            \begin{split}
                \frac{\d}{\dt} \sE[\bE_\bu, \bE_\bm, \bE_\bh] & + \sD[\bE_\bu, \bE_\bm, \bE_\bh]
                \leq \mu_0 \|\nabla \bh_a\|_{\dL^\infty} \|\bE_\bm\| \|\bE_\bu\| + \|\nabla \obU\|_{\dL^\infty} \|\bE_\bu\|^2 \\
                & \quad + \mu_0 \|\nabla \obH\|_{\dL^\infty} \|\bE_\bm\| \|\bE_\bu\| + \mu_0 \|\nabla \obM\|_{\dL^\infty} \|\bE_\bu\| \|\bE_\bh\| \\
                & \quad + \chi \|\nabla \obU\|_{\dL^\infty} (\|\bE_\bm\| + \|\bE_\bh\|) \|\bE_\bh\| + \|\nabla \obM\|_{\dL^\infty} \|\bE_\bu\| \|\bE_\bm\| \\
                & \quad + \chi \|\nabla \obM\|_{\dL^\infty} \|\bE_\bu\| \|\bE_\bh\| + \frac{1}{2 \nu} \|\sM_\bu \osfU\|_{\bbU'}^2 + \frac{\chi \mu_0}{2 \tau} \|\sM_\bm \osfU\|^2 \\
                & \leq \eta \left( 3 \|\bE_\bm\| \|\bE_\bu\| + \|\bE_\bu\|^2 + 2 \|\bE_\bu\| \|\bE_\bh\| + \|\bE_\bm\| \|\bE_\bh\| + \|\bE_\bh\|^2 \right) \\
                & \quad + \frac{1}{2 \nu} \|\sM_\bu \osfU\|_{\bbU'}^2 + \frac{\chi \mu_0}{2 \tau} \|\sM_\bm \osfU\|^2.
            \end{split}
            \]
            where $\eta = \max(\mu_0 \|\nabla \bh_a\|_{\dL^\infty}, (1 + \chi) \|\nabla \obU\|_{\dL^\infty}, \mu_0 \|\nabla \obH\|_{\dL^\infty}, (1 + \chi + \mu_0) \|\nabla \obM\|_{\dL^\infty})$. Thus,
            \[
            \begin{split}
                \frac{\d}{\dt} \sE[\bE_\bu, \bE_\bm, \bE_\bh] + \sD[\bE_\bu, \bE_\bm, \bE_\bh] & \leq 4 \eta \left( \|\bE_\bm\|^2 + \|\bE_\bu\|^2 + \|\bE_\bh\|^2 \right) \\
                & \quad  + \frac{1}{2 \nu} \|\sM_\bu \osfU\|_{\bbU'}^2 + \frac{\chi \mu_0}{2 \tau} \|\sM_\bm \osfU\|^2.
            \end{split}
            \]
            Integrating in time yields
            \[
            \begin{split}
                \sE& [\bE_\bu, \bE_\bm, \bE_\bh](t) + \int_0^t \sD[\bE_\bu, \bE_\bm, \bE_\bh] \ds \\
                & \leq 4 \int_0^t \eta(s) \left( \|\bE_\bm\|^2 + \|\bE_\bu\|^2 + \|\bE_\bh\|^2 \right) \ds + \int_I \frac{1}{2 \nu} \|\sM_\bu \osfU\|_{\bbU'}^2 + \frac{\chi \mu_0}{2 \tau} \|\sM_\bm \osfU\|^2 \ds.
            \end{split}
            \]
            Using Gr\"onwall's inequality, we obtain
            \[
            \begin{split}
                & \sE [\bE_\bu,\bE_\bm, \bE_\bh](t) + \int_0^t \sD[\bE_\bu, \bE_\bm, \bE_\bh] \ds \\
                & \leq \left( \sE[\bE_\bu, \bE_\bm, \bE_\bh](0) + \int_I \frac{1}{2 \nu} \|\sM_\bu \osfU\|_{\bbU'}^2 + \frac{\chi \mu_0}{2 \tau} \|\sM_\bm \osfU\|^2 \ds \right) \exp(C \|\eta\|_{L^1(I)}),
            \end{split}
            \]
            where $C = C(\mu_0, \chi, \tau)$. Finally, using $\sfe = \sfe^p + \sfE$ and triangle inequality, we deduce the asserted result.
        \end{proof}
        Estimate \eqref{eq:err-est} expresses the Galerkin error $(\be_\bu, \be_\bm, \be_\bh)$ in terms of the initial error $\sE_0$ and the consistency error $\sG$ of \eqref{eq:init-consist-errs}. Our next task is to quantify the latter provided $\sfu$ is a smooth solution.

    \subsection{Error estimate using the Stokes projection} \label{sec:err-est-stokes-proj}
        We first introduce suitable projections into the discrete spaces. We start with the \textit{Stokes projection} $\cS: \bH^1_0(\Omega) \times L^2_\#(\Omega) \to \bbU \times \bbP$, where $(\obU, \oP) := \cS(\bu, p)$ satisfies
        \begin{equation} \label{eq:stokes-proj}
            \nu \big(\nabla [\obU - \bu], \nabla \bV_\bu \big) - \big(\oP - p, \div \bV_\bu \big) - \big(\div [\obU - \bu], V_p \big) = 0, \quad \forall (\bV_\bu, V_p) \in \bbU \times \bbP.
        \end{equation}
        Due to (\textbf{A2}), if $\div \bu = 0$, then $\div \obU = 0$. Let $H^1_\#(\Omega) := H^1(\Omega) \cap L^2_\#(\Omega)$. We next define the \textit{Ritz projections} $\cR_\# : H^1_\#(\Omega) \to \bbX_\#, \cR_0 : H^1_0(\Omega) \to \bbX_0$ to be
        \begin{equation} \label{eq:ritz-proj}
            \oPhi := \cR_\# \phi : \quad \big(\nabla [\oPhi - \phi], \nabla V_\phi \big) = 0, \qquad \oPsi := \cR_0 \psi : \quad \big(\nabla [\oPsi - \psi], \nabla V_\psi \big) = 0,
        \end{equation}
        for all $V_\phi \in \bbX_\#$ and all $V_\psi \in \bbX_0$. We finally set 
        \[
        \obH = \nabla \oPhi \in \nabla \bbX_\#, \qquad \obM = \bcurl \oPsi - \nabla \oPhi \in \nabla \bbX_\# \oplus \bcurl \bbX_0.
        \]
        We further choose the initial conditions to be $\bU_0 = \obU(0)$, $\bH_0 = \obH(0)$, $\bM_0 = \obM(0)$, whence the initial error defined in \eqref{eq:init-consist-errs} satisfies $\sE_0 = 0$. The error expression then becomes
        \[
        \begin{split}
            & \sE [\be_\bu, \be_\bm, \be_\bh](t) + \int_0^t \sD[\be_\bu, \be_\bm, \be_\bh] \ds = 2 \sE[\be_\bu^p, \be_\bm^p, \be_\bh^p](t) \\
            & \quad + 2 \int_0^t \sD[\be_\bu^p, \be_\bm^p, \be_\bh^p] \ds + 2 \left(\int_I \frac{1}{2 \nu} \|\bD_\bu\|^2 + \frac{\chi \mu_0}{2 \tau} \|\bD_\bm\|^2 \ds \right) \exp(C \|\eta\|_{L^1(I)}),
        \end{split}
        \]
        where defects $\bD_\bu, \bD_\bm$ are given as
        \[
        \begin{split}
            \bD_\bu & = \pa_t (\bu - \obU) + (\bu \cdot \nabla) \bu - (\obU \cdot \nabla) \obU - \mu_0 (\bm \cdot \nabla) (\bh + \bh_a) + \mu_0 (\obM \cdot \nabla) (\obH + \bh_a), \\
            \bD_\bm & = \pa_t (\bm - \obM) + (\bu \cdot \nabla) \bm - (\obU \cdot \nabla) \obM - \tau^{-1} (\bm - \chi \bh) + \tau^{-1} (\obM - \chi \obH),
        \end{split}
        \]
        because $(\bu, \bm, \bh)$ satisfy the momentum and magnetization equations. Moreover, the diffusion term corresponding to $\nu (\nabla \obU, \nabla \bV_\bu)$ and the pressure term $(\oP, \div \bV_\bu)$ are gone thanks to using the Stokes projection. This is an instance of the standard Wheeler's argument for parabolic equations \cite{Wheeler1973}. Estimating the two defect terms yields
        \[
        \begin{split}
            \|\bD_\bu\| & \leq \|\pa_t \be_\bu^p\| + \|\be_\bu^p\| \|\nabla \obU\|_{\dL^\infty(\Omega)} + \|\bu\|_{\bL^\infty(\Omega)} \|\nabla \be_\bu^p\| \\
            & \quad + \mu_0 \|\be_\bm^p\| \|\nabla (\obH + \bh_a)\|_{\dL^\infty(\Omega)} + \mu_0 \|\bm\|_{\bL^\infty(\Omega)} \|\nabla \be_\bh^p\| \\
            & \leq C_1 (\|\pa_t \be_\bu^p\| + \|\nabla \be_\bu^p\| + \mu_0 \|\be_\bm^p\| + \mu_0 \|\nabla \be_\bh^p\|),
        \end{split}
        \]
        and
        \[
        \begin{split}
            \|\bD_\bm\| & \leq \|\pa_t \be_\bm^p\| + \|\be_\bu^p\| \|\nabla \obM\|_{\dL^\infty(\Omega)} + \|\bu\|_{\bL^\infty(\Omega)} \|\nabla \be_\bm^p\| + \tau^{-1} \|\be_\bm^p\| + \chi \tau^{-1} \|\be_\bh^p\| \\
            & \leq C_2 (\|\pa_t \be_\bm^p\| + \|\be_\bu^p\| + \|\nabla \be_\bm^p\| + \tau^{-1} \|\be_\bm^p\| + \chi \tau^{-1} \|\be_\bh^p\|),
        \end{split}
        \]
        where constants $C_1, C_2 > 0$ depend on the $L^\infty$ and $W^{1,\infty}$ operator norms of the Stokes projection restricted to the space of div-free vector fields, the $W^{1,\infty}$ and $W^{2,\infty}$-operator norms of the Ritz projections, and the $W^{1,\infty}$-norms of the exact solution $(\bu, \bm, \bh)$ and data $\bh_a$. Altogether, we deduce the following abstract error estimate that only has projection errors on the right-hand side:
        \begin{equation} \label{eq:err-est-stokes-proj}
            \begin{split}
                \sE & [\be_\bu, \be_\bm, \be_\bh](t) + \int_0^t \sD[\be_\bu, \be_\bm, \be_\bh] \ds \\
                & \leq 2 \sE[\be_\bu^p, \be_\bm^p, \be_\bh^p](t) + 2 \int_0^t \sD[\be_\bu^p, \be_\bm^p, \be_\bh^p] \ds \\
                & \quad + C_\eta \frac{4 C_1^2}{\nu} \int_I (\|\pa_t \be_\bu^p\|^2 + \|\nabla \be_\bu^p\|^2 + \mu_0^2 \|\be_\bm^p\|^2 + \mu_0^2 \|\nabla \be_\bh^p\|^2) \ds \\
                & \quad + C_\eta \frac{5 C_2^2 \chi \mu_0}{\tau} \int_I (\|\pa_t \be_\bm^p\|^2 + \|\be_\bu^p\|^2 + \|\nabla \be_\bm^p\|^2 + \tau^{-2} \|\be_\bm^p\|^2 + \chi \tau^{-2} \|\be_\bh^p\|^2) \ds,
            \end{split}
        \end{equation}
        where $C_\eta := \exp(C \|\eta\|_{L^1(I)})$.

    \subsection{Examples of discretization schemes} \label{sec:err-analysis}
        Assume that $\Omega \subset \mathbb{R}^2$ is convex and polygonal. Let $\{\cT_h\}_h$ be a family of shape-regular and quasi-uniform triangulations of $\Omega$ into unions of triangles of mesh size $h$.        
        For the velocity-pressure pair $(\bbU, \bbP)$, one can take the Scott--Vogelius \cite{ScottVogelius1985} or one of the Guzm\'an--Neilan \cite{GuzmanNeilan2018} elements. For the space $\bbX$, one can use the Argyris \cite{Argyris1968}, the Hsieh--Clough--Tocher \cite{CloughTocher1965} or the reduced Hsieh--Clough--Tocher \cite{Ciarlet1978-rHCT} elements, for instance.
        % To simplify the analysis, we assume that $\bbU$ consists of piecewise polynomials of degree $k-1$, where $k \geq 2$ is the polynomial order of $\bbX$.
        To simplify this analysis, we assume that in terms of approximation properties, $\bbU$ behaves like the space of piecewise polynomials of degree $\leq k-1$ and $\bbX$ like polynomials of degree $\leq k$.
        The following results are nonstandard and will, for the purposes of this work, be stated as assumptions.
        
        \begin{assump}[stability of projections] \label{assump:stab}
            There exist constants $C_3, C_4, C_5 > 0$ independent of mesh size $h$ such that
            \[
            \sup_{0 \neq v \in W^{2, \infty}(\Omega) \cap L^2_\#(\Omega)} \frac{\|\cR_\# v\|_{W^{2, \infty}(\Omega)}}{\|v\|_{W^{2,\infty}(\Omega)}} \leq C_3, \quad \sup_{0 \neq v \in W^{2, \infty}(\Omega) \cap H^1_0(\Omega)} \frac{\|\cR_0 v\|_{W^{2, \infty}(\Omega)}}{\|v\|_{W^{2, \infty}(\Omega)}} \leq C_4,
            \]
            \begin{equation}\label{eq:w1inf-stokes-proj}
            \sup_{\substack{\bv \in \bW^{1,\infty}(\Omega) \cap \bH^1_0(\Omega)\\ \div \bv = 0}} \frac{\|\nabla \cS_1 \bv\|_{\dL^\infty(\Omega)}}{\|\nabla \bv\|_{\dL^\infty(\Omega)}} \leq C_5.
            \end{equation}
            Here, $ \cS = (\cS_1,\cS_2) : \bH^1_0(\Omega) \times L^2_\#(\Omega) \to \bbU \times \bbP$ is the Stokes projection defined in \eqref{eq:stokes-proj}, and $\cR_\# : H^1_\#(\Omega) \to \bbX_\#, \cR_0 : H^1_0(\Omega) \to \bbX_0$ are the Ritz projections given in \eqref{eq:ritz-proj}.
        \end{assump}
        The $W^{2,\infty}$-stability of $\cR_\#$ and $\cR_0$ can be deduced from the $W^{1,\infty}$-stability and global inverse inequalities.
        The $\bW^{1,\infty} \times L^\infty$-stability of the full Stokes projection $\cS$ is a delicate matter that has been addressed in \cite{DuranNochetto1990, GiraultNochettoScott2015} for many popular inf-sup stable velocity-pressure pairs, which, however, do not directly include exactly divergence-free variants needed for our method.
        Fortunately, stability estimate \eqref{eq:w1inf-stokes-proj} is strictly weaker than the full stability of $\cS$ because it does not involve pressure.

        Equivalent characterizations of the velocity component $\cS_1$ of the Stokes projection $\cS$ and the stability bound \eqref{eq:w1inf-stokes-proj} in terms of an appropriate $H^2(\Omega)$-type projection can be given if $\bbU$ belongs to a Stokes-like discrete exact complex (see, e.g.~\cite[eq.~(4.13)]{JLMNR2017}
        \begin{equation*}
        \bbR
        \longrightarrow \Sigma
        \stackrel{\bcurl}{\longrightarrow} \bbU
        \stackrel{\div}{\longrightarrow} Y
        \longrightarrow 0,
        \end{equation*}
        where $\Sigma \subseteq H^2_0(\Omega)$.
        Namely, for $\bv \in \bW^{1,\infty}(\Omega) \cap \bH^1_0(\Omega)$ such that $\div \bv = 0$, there exists $\omega \in W^{2,\infty}(\Omega) \cap H^2_0(\Omega)$ such that $\bv = \bcurl \omega$.
        Consequently, $\cS_1 \bv = \bcurl (\cC \omega)$ where $\cC\omega \in \Sigma$ is the unique solution of
        \begin{equation*}
        (\nabla^2 \cC\omega, \nabla^2 \tau) = (\nabla^2 \omega, \nabla^2 \tau), \qquad \forall \tau \in \Sigma.
        \end{equation*}
        Consequently, estimate \eqref{eq:w1inf-stokes-proj} is equivalent to the $W^{2,\infty}(\Omega)$-stability of $\cC$:
        \begin{equation}\label{eq:w2inf-bihessian}
        \sup_{\omega \in W^{2,\infty}(\Omega) \cap H^2_0(\Omega)} \frac{\|\nabla^2 \cC \omega\|_{\dL^\infty(\Omega)}}{\|\nabla^2 \omega\|_{\dL^\infty(\Omega)}} \leq C_5.
        \end{equation}
        Such a statement is a natural extension of the $W^{1,\infty}$-stability of the Ritz projection for the Poisson problem (e.g.~\cite{RannacherScott1982}) to the biharmonic problem with clamped boundary condition.
        Surprisingly, we could not find a precise reference for this statement.

        \begin{assump}[approximability of projections] \label{assump:approx}
            There exist constants $C_6, C_7, C_8 > 0$ independent of mesh size $h$ such that for all $v \in H^{k+1}(\Omega) \cap L^2_\#(\Omega)$, $w \in H^{k+1}(\Omega) \cap H^1_0(\Omega)$, $\bv \in \bH^k_0(\Omega)$, and $k \in \mathbb{Z}_{\geq 2}$,
            \[
            \begin{split}
                \|v - \cR_\# v\| + h \|\nabla (v - \cR_\#) v\| + h^2 \|\nabla \nabla (v - \cR_\# v)\| & \leq C_6 h^{k+1} |v|_{H^{k+1}(\Omega)}, \\
                \|w - \cR_0 w\| + h \|\nabla (w - \cR_0) w\| + h^2 \|\nabla \nabla (w - \cR_0 w)\| & \leq C_7 h^{k+1} |w|_{H^{k+1}(\Omega)}, \\
                \|\bv - \cS \bv\| + h \|\nabla (\bv - \cS) \bv\| & \leq C_8 h^k |\bv|_{\bH^k(\Omega)}.
            \end{split}
            \]
        \end{assump}
        The error estimate for $\cR_\#$ and $\cR_0$ can be derived using standard arguments. In turn, the error estimate for $\cS$ may require further mesh restrictions, depending on the choice of velocity-pressure pair.
        
        \begin{theorem}[error estimate] \label{thm:err-est-final}
            Under Assumptions \ref{assump:stab} and \ref{assump:approx}, the following error estimate is valid for all $t \in (0, T]$ provided the exact solution is sufficiently smooth:
            \[
            \sE [\be_\bu, \be_\bm, \be_\bh](t) + \int_0^t \sD[\be_\bu, \be_\bm, \be_\bh] \ds \leq C h^{2(k-1)} \Lambda(\bu, \bm, \bh), 
            \]
            where
            \[
            \begin{split}
                & \qquad \qquad \qquad \quad \Lambda(\bu, \bm, \bh) := \|\bu\|_Y^2 + \|\bm\|_Y^2 + \|\bh\|_X^2, \\
                \|\cdot\|_{X}^2 & := \|\cdot\|_{L^\infty(I; \bH^{k-1}(\Omega))}^2 + \|\cdot\|_{L^2(I; \bH^k(\Omega))}^2, \ \  \|\cdot\|_Y^2 := \|\cdot\|_X^2 + \|\pa_t \cdot\|_{L^2(I; \bH^{k-1}(\Omega))}^2,
            \end{split}
            \]
            $k \in \mathbb{Z}_{\geq 2}$, and constant $C > 0$ depends on parameters $\mu_0, \chi, \tau, \nu$, final time $T$, constants from Assumptions \ref{assump:stab}, \ref{assump:approx}, and the $L^\infty(I; \bW^{1,\infty}(\Omega))$-norms of $\bh_a, \bu$, $\bh$, and $\bm$.
        \end{theorem}
        \begin{proof}
            We combine the abstract estimate \eqref{eq:err-est-stokes-proj} with Assumption \ref{assump:approx}.
            % We omit further details.
        \end{proof}
        The estimate of Theorem \ref{thm:err-est-final} is optimal with respect to order-regularity in space in the $L^2(I; \bH^1(\Omega))$-norm for velocity $\bu$ and polynomial degree $k-1$. Due to the nonlinearities, for $(\bh, \bm)$ the rate of convergence is one order lower than the corresponding interpolation estimate for $L^\infty(I; L^2(\Omega))$-norm and polynomial degree $k-1$ (variables $\bM$ and $\bH$ are first derivatives of elements in $\bbX$, which is of polynomial degree $k$).
        Whether this rate is sharp for the problem remains open, though our numerical experiments indicate that it might not be the case; see Table \ref{tab:spatial_convergence}.

\section{Computational results} \label{sec:experiment}
In this section, we present two sets of numerical experiments illustrating the performance of our 2D scheme \eqref{eq:semidiscrete}.
Our implementation uses the open-source finite element software \texttt{Firedrake} \cite{FiredrakeUserManual}.

Recall that $\{\cT_h\}_{h>0}$ denotes a family of shape-regular and quasi-uniform triangulations of $\Omega\subset\mathbb{R}^2$ made of triangles $T$ and that $h = \max_{T \in \cT_h} h_T$.
For each $h>0$ we denote by $\cT_h^\b$ the barycentric refinement of $\cT_h$.
For the velocity-pressure pair $(\bbU,\bbP)$ we consider the lowest order Guzmán--Neilan element \cite[Section 6.2]{GuzmanNeilan2018}, with $\bbU = \bbU(\cT_h^b)$ consisting of piecewise continuous vector-valued polynomials of degree $\leq 1$ enriched with ``modified'' face bubbles, and $\bbP = \bbP(\cT_h^\b)$ consisting of piecewise constant scalar-valued functions.
In turn, we let $\bbX = \bbX(\cT_h)$ be the space of reduced Hsieh--Clough--Tocher functions on $\cT_h$ \cite{Ciarlet1978-rHCT}.
Accordingly, as specified in \eqref{eq:Xsubspaces}, we let $\bbX_\# := \bbX \cap L^2_\#(\Omega)$ and $\bbX_0 := \bbX \cap H^1_0(\Omega)$.
The triplet $(\bbU,\bbP,\bbX)$ satisfies the discrete inf-sup condition \eqref{eq:disc-inf-sup} and assumptions A1 and A2 from section \ref{sec:spaces}.
We expect Assumptions \ref{assump:stab} and \ref{assump:approx} to hold true, but we are unable to pinpoint a specific reference for them.

For the time discretization of system \eqref{eq:semidiscrete}, we implement an Euler scheme in which all linear terms are treated implicitly, and the nonlinear convective terms are treated semi-implicitly.
Regarding the vanishing Dirichlet condition of $\Psi \in \bbX_0$, it is well-known that, for $C^1$-conforming elements, strongly enforcing Dirichlet boundary conditions without unnecessarily restricting or locking other degrees of freedom is a rather nontrivial task \cite{KirbyMitchell2019}.
Consequently, within each time step, we resort to a Nesterov acceleration of the Uzawa--Powell algorithm \cite[Section 8]{BenziGolubLiesen2005} to ensure that $\|\Psi\|_{L^2(\partial\Omega)}$ is negligible compared to time and spatial discretization parameters.

Given the finite element spaces involved in our approach, and the accelerated Uzawa--Powell iteration used to enforce zero Dirichlet condition for potential $\Psi$, it is clear that our algorithm is rather complicated compared to the ones found in the literature.
On the other hand, under non-obvious but reasonable assumptions, the present approach has theoretical guarantees for the \textit{unregularized} ferrofluid system \eqref{eq:ff}, i.e., with the magnetization equation retaining the transport term $(\bu \cdot \nabla) \bm$ and not requiring a regularization term $\sigma \Delta \bm$. To the best of our knowledge, the present approach is the only one that has this property. \\
\\
\noindent\textbf{Convergence test.}
In this test, we experimentally corroborate the spatial convergence rates asserted by Theorem \ref{thm:err-est-final}.
We set $\nu = 1$, $\mu_0 = 10^{-6}$, $\tau = 10^{-5}$ and $\chi = 0.1$.
The computational domain $\Omega = (0,1)^2$ is triangulated with uniform triangles of diameter $h$.
We consider $T = 0.5$ and set the time step size as $\Delta t = h^2/16$.
Within each time step, we halt the Uzawa--Powell--Nesterov inner scheme as soon as $\|\Psi\|_{L^2(\partial\Omega)} < 10^{-4} h^3$.

We consider the applied magnetic field $\bh_a = \nabla\phi_a$ given by the harmonic potential $\phi_a(t,x,y) = (t+1)^3 (x^6 - 15x^4 y^2 + 15x^2 y^4 - y^6 + 30x + 30y)$.
The manufactured solutions are given by
\begin{multline*}
\bu(t,x,y) = e^{-t} \bcurl( x^2 (1-x)^2 y^2), \qquad
p(t,x,y) = (x-1)^3 \sin(2\pi\, t\, y) - C_1(t),\\
\phi(t,x,y) = (1-t)^2 \cos(\pi x)\cos(\pi y) - C_2(t), \qquad
\psi(t,x,y) = x^3(1-x)^2 y(1-y).
\end{multline*}
where $C_1$ and $C_2$ are such that $\int_\Omega p \dx= \int_\Omega \phi \dx = 0$ for each $t \in [0,T]$.
The magnetic H-field and magnetization field are respectively given by $\bh = \nabla \phi$ and $\bm = \bcurl\psi - \nabla\phi$.
We modify the momentum and magnetization equations by incorporating appropriate forcing terms so as to ensure that the manufactured solutions satisfy the equations.
The nonzero Dirichlet condition satisfied by $\bu$ is incorporated in the system in the traditional way;
we note that this is not strictly covered by our theory.

As usual, for two consecutive meshes of sizes $h$ and $\hat h$, and corresponding generic error norms $\|\sfe(h)\|$ and $\|\sfe(\hat h)\|$, we let $\mathsf{r} = \log(\|\sfe(h)\|/\|\sfe(\hat h)\|)/\log(h/\hat h)$ be the experimental rate of convergence.
Our numerical results are summarized in Table \ref{tab:spatial_convergence}.\\
\begin{table}%[hbtp]
    \centering
    \caption{
    Spatial convergence rates for the fully discrete scheme.
    The rate for $\bu$ in the $L^2(I;\bH^1)$-norm agrees with Theorem \ref{thm:err-est-final}, but the other norms are better than expected.
    }
    \label{tab:spatial_convergence}
    \begin{tabular}{c c c c c c c c c}
        \toprule
        $h$ & $\|\bu - \bU\|_{L^\infty(L^2)}$ & Rate & $\|\bu-\bU\|_{L^2(H^1)}$ & Rate & $\|\bh-\bH\|_{L^\infty(L^2)}$ & Rate & $\|\bm-\bM\|_{L^\infty(L^2)}$ & Rate \\
        \midrule
        0.3536 & 8.0612\texttt{E}-03 & --   & 1.6236\texttt{E}-01 & --   & 4.6415\texttt{E}-02 & --   & 4.6439\texttt{E}-02 & --   \\
        0.1768 & 2.0876\texttt{E}-03 & 1.95 & 6.8749\texttt{E}-02 & 1.24 & 1.0785\texttt{E}-02 & 2.11 & 1.0793\texttt{E}-02 & 2.11 \\
        0.0884 & 5.0670\texttt{E}-04 & 2.04 & 3.0905\texttt{E}-02 & 1.15 & 2.6159\texttt{E}-03 & 2.04 & 2.6179\texttt{E}-03 & 2.04 \\
        0.0442 & 1.2142\texttt{E}-04 & 2.06 & 1.4633\texttt{E}-02 & 1.08 & 6.4704\texttt{E}-04 & 2.02 & 6.4754\texttt{E}-04 & 2.02 \\
        0.0221 & 2.9265\texttt{E}-05 & 2.05 & 7.1297\texttt{E}-03 & 1.04 & 1.6124\texttt{E}-04 & 2.00 & 1.6136\texttt{E}-04 & 2.00 \\
        \bottomrule
    \end{tabular}
\end{table}

% \subsection{Simulation: ferrofluid pumping}
\noindent\textbf{Simulation: ferrofluid pumping.}
In this experiment we illustrate the capability of our numerical scheme of representing complex phenomena involving ferrofluids.
One of the most prominent applications of ferrofluids is pumping induced by an applied magnetic field.
Highly inspired by \cite[Section 7.2]{NochettoSalgadoTomas2016}, we consider a spatial and time dependent applied magnetic field $\bh_a$ given by the sum of a finite number of periodic pulses that travel in the direction that we would like to induce linear momentum \cite{MaoKoser2005}.
More precisely, we consider a long channel $\Omega = (0,6)\times (0,1)$ as our computational domain, partitioned with uniform triangles of size $h = \sqrt{2}/32$.
The computational time is $T = 3.65$ with uniform time steps of size $\Delta t = h^2/16$.
Physical parameters are given by $\nu = 0.1$, $\mu_0 = 10^{-6}$, $\tau = 10^{-5}$ and $\chi = 0.5$.
The magnetic potential $\phi_a$ defining the applied magnetic field $\bh_a = \nabla\phi_a$ is given by the sum of 64 dipoles with time-dependent intensity placed outside of $\Omega$.
More precisely, for $\lambda = 1$, $\kappa = 2\pi/\lambda$, $q = 5$, $f = 10$, $\omega = 2\pi f$, $\bd^{\texttt{top}} = (0,1)$, and $\bd^{\texttt{bot}} = (0,-1)$, we set 
\begin{equation*}
\textstyle
\phi_a(t,\bx)
\textstyle
= R(t) \sum\limits_{i=1}^{32} \Big[ \sin(\omega t - \kappa x_i^{\texttt{top}})^{2q} \bd^{\texttt{top}} \cdot \frac{\bx - \bx_i^{\texttt{top}}}{\|\bx - \bx_i^{\texttt{top}}\|^2}
\textstyle
+ \sin(\omega t - \kappa x_i^{\texttt{bot}})^{2q} \bd^{\texttt{bot}} \cdot \frac{\bx - \bx_i^{\texttt{bot}}}{\|\bx - \bx_i^{\texttt{bot}}\|^2} \Big];
\end{equation*}
where $(x_i^{\texttt{top}}, y_i^{\texttt{top}}) = \bx_i^{\texttt{top}} = (2 + \frac{2}{31}(i-1), 1.05)$ and $(x_i^{\texttt{bot}}, y_i^{\texttt{bot}}) = \bx_i^{\texttt{bot}} = (2 + \frac{2}{31}(i-1), -0.05)$ and $R$ is the ``ramp'' function defined by $R(t) = \max\{f t , 1\}$.
We enforce the no-slip boundary condition for the velocity field $\bU$ at the top and bottom of the channel, and the traction-free Neumann condition at the entrance and exit of the channel.
All variables are homogeneously initialized.
We report some of our findings in
% Figures  \ref{fig:H-field}, \ref{fig:backward-pumping} and \ref{fig:velocity-final-time}
Figures \ref{fig:H-field}, \ref{fig:backward-pumping} and \ref{fig:velocity-final-time}.
\begin{figure}
    \centering
    \includegraphics[width=0.9\linewidth]{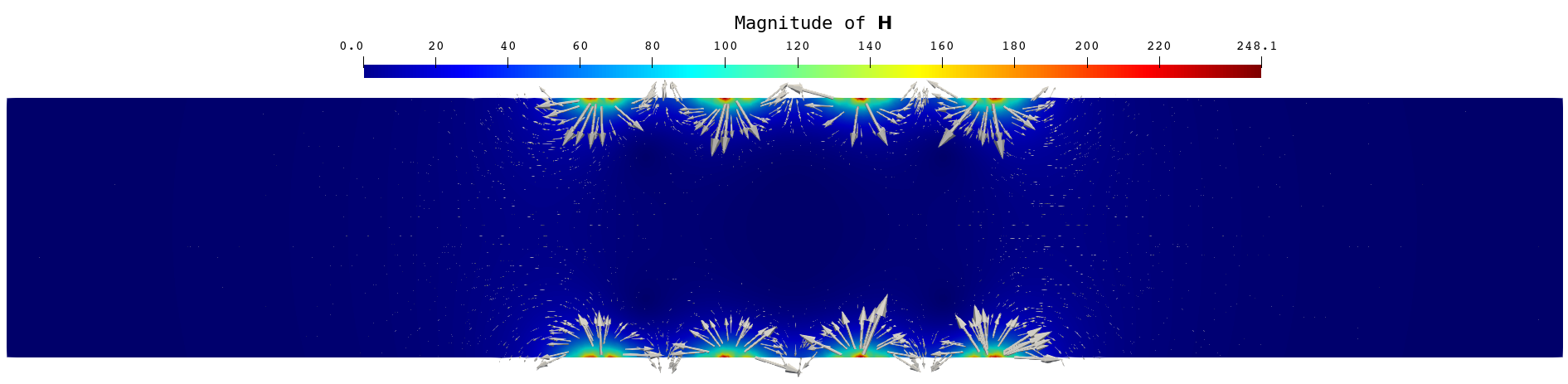}
    \caption{
    \emph{Magnetic H-field at time $t=0.1$}.
    The magnetic H-field $\bH$ (and the magnetization field $\bM$ as well) rapidly aligns with the applied magnetic field $\bh_a$.
    }
    \label{fig:H-field}
\end{figure}
\begin{figure}
    \centering
    \includegraphics[width=0.6\linewidth]{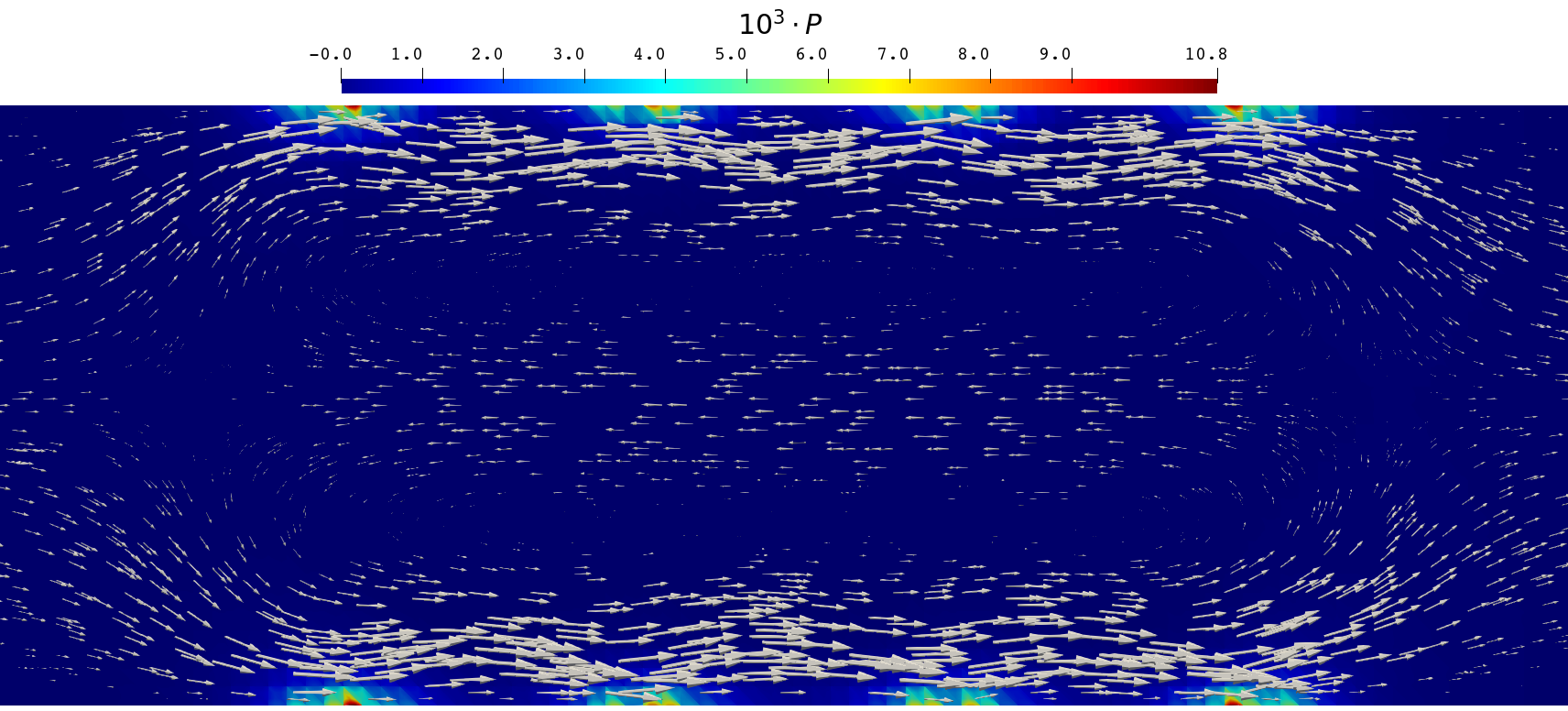}
    \caption{\emph{Backward pumping anomaly: velocity and pressure fields in the middle of the channel at time $t = 0.5$.}
    The pumping process induce a rather counterintuitive backward flow in the interjacent region between the top and bottom magnetic dipoles.
    This is a well-documented phenomenon \cite{RinaldiZahn2002}.
    }
    \label{fig:backward-pumping}
\end{figure}
\begin{figure}
    \centering
    \includegraphics[width=0.9\linewidth]{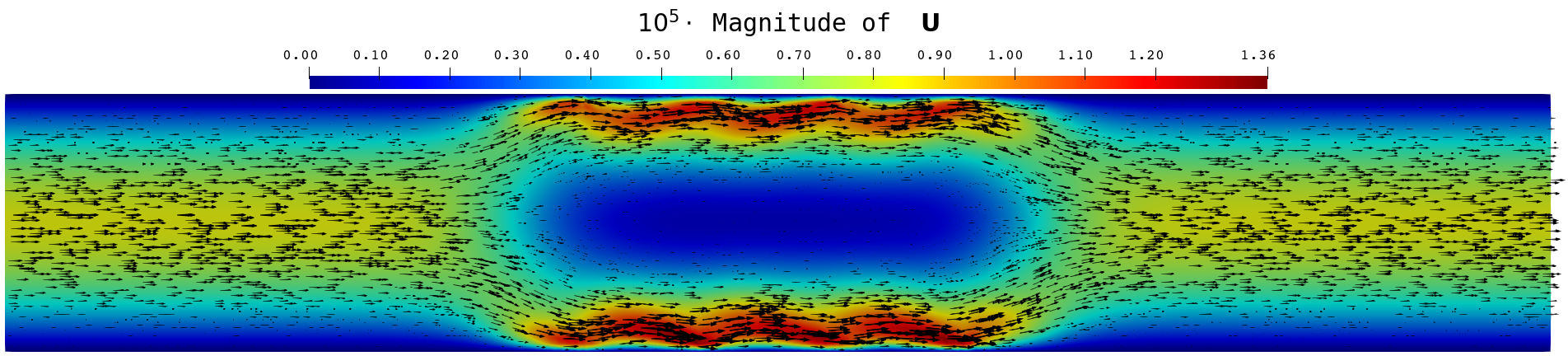}
    \caption{
    \emph{Velocity field at final time $T=3.65$}.
    The magnetic pulse succeeded at inducing linear momentum, and consequently a bulk flow, from left to right of the channel.
    The previously observed backward pumping has completely dissipated.
    }
    \label{fig:velocity-final-time}
\end{figure}

\begin{acknowledgement}
    We thank Franziska Weber (UC Berkeley) for helpful discussions regarding the ferrofluids system, and Johnny Guzm\'an (Brown University) for suggesting to use two scalar potentials.
\end{acknowledgement}
\ethics{Competing Interests}{All three authors were partially supported by the NSF Grant DMS-2512392. The authors have no conflicts of interest to declare that are relevant to the content of this work.}

% \section*{Appendix}
% \addcontentsline{toc}{section}{Appendix}
% %
% %
% When placed at the end of a chapter or contribution (as opposed to at the end of the book), the numbering of tables, figures, and equations in the appendix section continues on from that in the main text. Hence please \textit{do not} use the \verb|appendix| command when writing an appendix at the end of your chapter or contribution. If there is only one the appendix is designated ``Appendix'', or ``Appendix 1'', or ``Appendix 2'', etc. if there is more than one.

% \begin{equation}
% a \times b = c
% \end{equation}

\printbibliography

\end{document}